\documentclass[12pt]{article}
\usepackage{amsmath,amssymb,amsthm,comment}
\usepackage{amsfonts}
\usepackage{color}
\usepackage[dvips]{graphicx}
\usepackage{epsfig}
\usepackage{here}
\newtheorem{theo}{Theorem}[section]
\newtheorem{lemm}[theo]{Lemma}

\newtheorem{rema}[theo]{Remark}
\newtheorem{prop}[theo]{Proposition}
\newtheorem{conj}[theo]{Conjecture}

\makeatletter
    
    \@addtoreset{equation}{section}
  \makeatother
\newcommand{\nn}{\nonumber}
\newcommand{\io}{\int_\Omega}

\begin{document}
\title{Global boundedness of solutions to a parabolic-parabolic chemotaxis system with local sensing in higher dimensions}
\author{
Kentaro Fujie\footnote{fujie@tohoku.ac.jp}\\
%{\small Faculty of Science Devision I,}\\
{\small Tohoku University,}\\
{\small Sendai, 980-8578, Japan}
\and
Takasi Senba\footnote{senba@fukuoka-u.ac.jp}\\
%{\small Faculty of Science,}\\
{\small Fukuoka University,}\\
{\small Fukuoka, 814-0180, Japan}
\medskip
}
\date{\small\today}
\maketitle
\begin{abstract}
This paper deals with classical solutions to the parabolic-parabolic system
\begin{align*}
\begin{cases}
	u_t=\Delta (\gamma (v) u )
		&\mathrm{in}\ \Omega\times(0,\infty), \\[1mm]
	v_t=\Delta v - v + u
	&\mathrm{in}\ \Omega\times(0,\infty), \\[1mm]
\displaystyle	\frac{\partial u}{\partial \nu} = \frac{\partial v}{\partial \nu} = 0 
	&\mathrm{on}\ \partial\Omega \times (0,\infty), \\[1mm]
	u(\cdot,0)=u_0, \ v(\cdot,0)=v_0 &\mathrm{in}\ \Omega,
\end{cases}
\end{align*}
where $\Omega$ is a smooth bounded domain in $\mathbf{R}^n$($n \geq 3$), $\gamma (v)=v^{-k}$ ($k>0$) and the initial data $(u_0,v_0)$ is positive and regular.  
This system has striking features similar to those of the logarithmic Keller--Segel system. 
It is established that classical solutions of the system exist globally in time and remain uniformly bounded in time if $k \in (0,n/(n-2))$, independently the magnitude of mass.  
This constant $n/(n-2)$ is conjectured as the optimal range guaranteeing global existence and boundedness in the corresponding logarithmic Keller--Segel system. 
We will derive sufficient estimates for solutions through some single evolution equation that some auxiliary function satisfies. 
The cornerstone of the analysis is the refined comparison estimate for solutions, which enables us to control the nonlinearity of the auxiliary equation. \\ \\
 {\bf Key words:} global existence; uniform boundedness; chemotaxis \\
 {\bf AMS Classification:} 35B45, 35K57, 35Q92, 92C17.
\end{abstract}
\newpage
%\tableofcontents
%
%
%
\section{Introduction}\label{section_introduction}
We consider the following fully parabolic system:
 \begin{align}
\label{eqn:BaseMod}
\begin{cases}
	u_t=\Delta (\gamma (v) u )
		&\mathrm{in}\ \Omega\times(0,\infty), \\[1mm]
	\tau v_t=\Delta v - v + u &\mathrm{in}\ \Omega\times(0,\infty), \\[1mm]
	\frac{\partial u}{\partial \nu} = \frac{\partial v}{\partial \nu} = 0 &\mathrm{in}\ \partial\Omega\times(0,\infty).   
\end{cases}
 \end{align}
Here $\Omega$ is a bounded domain of $\mathbf{R}^n$ ($n\geq 1$) with smooth boundary, the relaxation time $\tau$ is a nonnegative constant, and $u$ and $v$ denote the density of cells and the concentration of a chemical substance, respectively. This system is a simplified system introduced in \cite{Science, Fuel} to describe the process of stripe pattern formation of cells. 
The function $\gamma (v)$ represents a signal-dependent motility function, which is positive and decreasing. 
This system is also introduced to describe chemotaxis movement taking account of the local sensing effect (\cite{KSb}).

From a mathematical view point, positive smooth solutions to this system have properties similar to those of the so-called Keller--Segel system:
 \begin{align}
\label{eqn:KellerSegel}
\begin{cases}
	u_t=\nabla \cdot \left( \nabla u - u \nabla \chi (v)\right)
		&\mathrm{in}\ \Omega\times(0,\infty), \\[1mm]
	\tau v_t=\Delta v - v + u &\mathrm{in}\ \Omega\times(0,\infty), \\ 
		\frac{\partial u}{\partial \nu} = \frac{\partial v}{\partial \nu} = 0 &\mathrm{in}\ \partial\Omega\times(0,\infty),
\end{cases}
 \end{align}
which is also introduced to describe the aggregation of cells.  
Here $\chi(v)$ is a signal-dependent sensitivity function. 
Since the first equation of (\ref{eqn:BaseMod}) can be rewritten as 
\[
u_t = \nabla \cdot \gamma (v) \left( \nabla u - u \nabla \log \frac{1}{\gamma (v)} \right), 
\]
stationary structures of both systems coincide (see \cite{FJ21}) if
$$
\chi (v) = \log \frac{1}{\gamma (v)}.
$$
Especially, in the special case where $\gamma (v) = e^{-v}$ and $\chi (v) =v$, (\ref{eqn:BaseMod}) and (\ref{eqn:KellerSegel}) have the Lyapunov functionals:
\begin{eqnarray*}
&&\frac{d}{dt} \mathcal{F}(u(t),v(t)) +\io |v_t|^2\,dx 
+ \int_\Omega ue^{-v} |\nabla (\log u - v)|^2 \,dx = 0 \quad (\mbox{see \cite{FJ21, JW}}) \\
\mbox{ and } && \\
&&\frac{d}{dt} \mathcal{F}(u(t),v(t)) +\io |v_t|^2\,dx 
+ \int_\Omega u |\nabla (\log u - v)|^2\, dx = 0 \quad (\mbox{see \cite{NSY}}),
\end{eqnarray*}
respectively, where
\[
 \mathcal{F}(u(t),v(t)) := \int_\Omega \left( u(t)\log u(t) - u(t)v(t) + \frac{1}{2}|\nabla v(t)|^2 + \frac{1}{2}v(t)^2 \right)\, dx.
\]

It is of interest to know whether the solutions of both systems behave similarly or not.  
In the case where $(\gamma (v), \chi (v) )= (e^{-v}, v)$ and $n=2$, by using these Lyapunov functionals and the mass conservation law, global existence and blowup of solutions are studied; 
Solutions to both systems \eqref{eqn:BaseMod} and \eqref{eqn:KellerSegel} exist globally in time and are uniformly bounded in time if $\int_\Omega u(x,0) dx < \theta^\ast$ (\cite{FJ21, JW, NSY}); 
there exist unbounded solutions to these systems if $\int_\Omega u(x,0) dx > \theta^\ast$ (\cite{FJ21,  JW, HW}),
where $\theta^\ast = 8\pi$ for radially symmetric solutions 
and $\theta^\ast =4\pi$ for nonradial solutions. 
Although \eqref{eqn:KellerSegel} has finite time blowup solutions (\cite{MizWin}), 
solutions of \eqref{eqn:BaseMod} exist globally in time and blow up at infinite time (\cite{FJ21}). 
Similar result is established in the higher dimensional setting ($n \geq 3$). 
Global existence of some weak solutions to \eqref{eqn:BaseMod} with $\gamma(v)=e^{-v}$ is established in \cite{BLT}. 
Moreover in \cite{FSpreprint} global existence of classical solutions to \eqref{eqn:BaseMod} is proved for general smooth and decreasing function $\gamma$ and infinite time blowup solutions are constructed when $\gamma (v) = e^{-v}$
despite the fact that \eqref{eqn:KellerSegel} with $\chi (v)=v$ has finite time blowup solutions even for small data in the higher dimensional setting (\cite{Win2013}).

In the present paper, we pay attention to the case where $\gamma (v) = v^{-k}$ and $\chi (v) = k \log v$ with $k>0$.  
In these cases, the first equations of (\ref{eqn:BaseMod}) and (\ref{eqn:KellerSegel}) can be rewritten as 
\[
u_t = \nabla\cdot v^{-k} \left( \nabla u - k \nabla \log v \right) 
\]
and
\[
u_t = \nabla\cdot \left( \nabla u - k \nabla \log v \right),
\]
respectively. 
As to the system \eqref{eqn:KellerSegel} in the higher dimensional setting $n \geq3$, it is conjectured that the threshold number separating global existence and blowup of classical solutions is $k=n/(n-2)$ in \cite{FS2018} based on observations on some nonlinear heat equation which can be seen as the limit of the system \eqref{eqn:KellerSegel}.
This conjecture is still open, however there are several results in support of this conjecture. 
Actually, if $k<\frac{n}{n-2}$ global existence and boundedness are established under some technical conditions on $\tau$ and domain (\cite{FS2018}) and in some special framework of weak solutions (\cite{LW}); 
finite time blow up solutions of the corresponding parabolic-elliptic system are constructed for $k>\frac{2n}{n-2}$ (\cite{NS98}). 
Concerning solutions to (\ref{eqn:BaseMod}), there are also many researches corresponding to this conjecture.  
Boundedness of solutions are established for small $k>0$ (for parabolic-parabolic case \cite{D2019,  FJ21, FJpreprint}; for parabolic-elliptic case \cite{Anh19,FJpreprint}), but unfortunately in these results the conditions on $k$ is strictly smaller than the conjectured range $(0,\frac{n}{n-2})$.  
\textcolor{black}{ In \cite{YK2017} global existence and boundedness are established for any $k>0$, but some smallness of another parameter is required, 
and this smallness condition restricts the size of the initial data (We will see details in Section \ref{sec:discussion}). }
Recently, Jiang and Lauren\c{c}ot \cite{JL} study the simplified system ($\tau =0$) and show the following:
\begin{itemize}
\item If $\gamma (v)$ is positive and bounded, then solutions of \eqref{eqn:BaseMod} with $\tau =0$ exist globally in time;
\item If $\gamma (v) = v^{-k}$ and $k \in(0, n/(n-2))$, then global solutions of \eqref{eqn:BaseMod} with $\tau=0$ are uniformly bounded in time. 
\end{itemize}

For the fully parabolic system, in \cite{FSpreprint} the authors show that solutions of \eqref{eqn:BaseMod} with $\tau >0$ exist globally in time if $\gamma$ satisfies the following condition
\begin{description}
\item[(A0)]  $ \gamma \in C^3((0,\infty)), \quad
 \gamma >0 \ \mbox{ in } (0,\infty),\quad
  \gamma^\prime \leq 0 \ \mbox{ on } (0,\infty)$.
\end{description}
In the present paper, our interest is in whether global solutions of \eqref{eqn:BaseMod} with $\tau>0$ and $\gamma (v) = v^{-k}$ remain bounded or blow up at infinite time for the conjectured range $k \in (0,n/(n-2))$, independently of the size of the initial data. 
That is, without any technical assumptions (any restriction on $\tau$) we will give an affirmative answer to the corresponding boundedness problem to the above conjecture. 

We consider the following initial boundary-value problem:
 \begin{align}
\label{eqn:kinemod}
\begin{cases}
	u_t=\Delta (\gamma (v) u )
		&\mathrm{in}\ \Omega\times(0,\infty), \\[1mm]
	\tau v_t=\Delta v - v + u
	&\mathrm{in}\ \Omega\times(0,\infty), \\[1mm]
\displaystyle	\frac{\partial u}{\partial \nu} = \frac{\partial v}{\partial \nu} = 0 
	&\mathrm{on}\ \partial\Omega \times (0,\infty), \\[1mm]
	u(\cdot,0)=u_0, \ v(\cdot,0)=v_0 &\mathrm{in}\ \Omega,
\end{cases}
 \end{align}
where $\tau>0$, $\Omega$ is a bounded domain in $\mathbf{R}^n$ ($n\geq 3$) with smooth boundary $\partial \Omega$.
Moreover, we assume the following:
\begin{description}
 \item[(A1)] $\gamma (v)=v^{-k}$ with some $k \in (0, \frac{n}{n-2})$.
 \item[(A2)]  $(u_0,v_0)\in (W^{1,p_0}(\Omega))^2 \;\text{with some}\;p_0>n,\quad u_0\geq0,\; v_0 > 0 \quad  \mbox{in } \overline\Omega, \quad u_0\not\equiv0$.
\end{description}
We note here that global existence of classical solutions is guaranteed since $\gamma (v)=v^{-k}$ with $k>0$ satisfies {\bf (A0)}. 
Our main result reads as follows.
\begin{theo}
 \label{theo:TG}
 Let $\tau>0$.
Suppose {\bf (A1)} and {\bf (A2)}.  
Then solutions of \eqref{eqn:kinemod} exist globally in time and 
remain bounded uniformly in time:
$$
\sup_{t\in [0,\infty)} \left( \|u(t)\|_{L^\infty(\Omega)} + \|v(t)\|_{L^{\infty}(\Omega)} \right)<\infty.
$$ 
\end{theo}

\begin{rema}
In \cite{FS2018} the corresponding result to the above is obtained for the logarithmic Keller--Segel system if the parameter $\tau>0$ is sufficiently large. 
Differently from this result, Theorem \ref{theo:TG} hods for any $\tau>0$. 
\end{rema}

%\cred{%
\begin{rema}
In Theorem \ref{theo:TG}, the condition {\bf (A1)} can be replaced by 
\begin{description}
 \item[(A1)$^\prime$] 
 $\gamma $ satisfies {\bf (A0)} and 
 $\frac{\Gamma_{min}}{ v^{k}} \leq \gamma (v)\leq \frac{\Gamma_{max}}{v^{k}}$ 
 with some $k \in (0, \frac{n}{n-2})$ and $\Gamma_{min}, \Gamma_{max}>0$.
\end{description}
 See Remark \ref{remA1replaced} for details.
\end{rema}
%}%

\begin{rema}
We expect that the threshold number separating boundedness and unboundedness of classical solutions to \eqref{eqn:kinemod} is $\frac{n}{n-2}$. We will give details of this conjecture in Section \ref{sec:discussion}. 
\end{rema}

The cornerstone of the present analysis is the following auxiliary identity, which can be derived after taking the resolvent of the Neumann operator $(1-\Delta_N)^{-1}$ to the first equation of \eqref{eqn:kinemod} :
\begin{equation}\label{oldview}
 w_t + \gamma (v) u = (1-\Delta_N)^{-1} (\gamma (v) u) \quad \mbox{ in } \Omega \times (0,\infty),
\end{equation}
where $w(t) :=(1-\Delta_N)^{-1}u(t)$.  
This identity is the specific feature of \eqref{eqn:kinemod} differently from the logarithmic Keller--Segel system \eqref{eqn:KellerSegel}. 
Indeed, in previous results (\cite{FJ20, FJ21,FJpreprint}) a priori estimates for $w$ and the comparison estimate
$$v(x,t) \leq Cw(x,t) \quad \mbox{for } (x,t)\in \Omega \times (0,\infty)$$
are established by using \eqref{oldview}. 
With these estimates in hand the standard iteration argument is applied for \eqref{eqn:kinemod} to derive sufficient regular estimates guaranteeing global existence and boundedness of solutions.  

In the present paper, on the other hand,  we change our perspective. 
By the definition of $w$, it follows
\begin{equation}\label{newview}
 w_t + \gamma (v) (\Delta w-w) = (1-\Delta_N)^{-1} (\gamma (v) u) \quad \mbox{ in } \Omega \times (0,\infty),
\end{equation}
which can be treat as an evolution equation of $w$.  
We first develop the refined comparison estimate in Lemma \ref{lemm:Estvw}:
$$ d_1 v(x,t) \leq w(x,t) \leq d_2 v(x,t) \quad \mbox{for } (x,t) \in \Omega \times (0,\infty)$$
with some constants $d_1,d_2>0$.  
Thanks to this estimate and the compound of the comparison principles, we can derive an upper estimate of the nonlinearity $(1-\Delta_N)^{-1} (\gamma (v) u) $ of \eqref{newview} in Lemma \ref{lemm:bounduvk} and then the testing argument implies $L^p$-estimates for $v$ and $w$ for any $p \in [1,\infty)$ in Lemma \ref{lemm:boundvLp}. Here we need the restriction $k <\frac{n}{n-2}$ coming from the optimal exponent of the regularity estimate of the elliptic equation in \cite{BS}. 
In view of these $L^p$-estimates, we can apply the iteration argument to the single equation \eqref{newview} to establish $L^\infty$-estimates for $v$ and $w$ (Proposition \ref{prop:boundvLinfty}).
By employing the abstract theory of evolution equations to \eqref{newview}, Theorem \ref{theo:TG} is proved in Section \ref{sectionproof}. 
Finally, Section \ref{sec:discussion} is devoted to the discussion about the threshold number.

\section{Preliminaries}
\label{sec:preliminaries}
We first recall global existence of the system \eqref{eqn:kinemod},  which is shown in \cite[Theorem 1.1]{FSpreprint}.
\begin{prop}
\label{prop:TGexist}
 Let $\tau>0$.
Suppose {\bf (A0)} and {\bf (A2)}.  
Then there exists a unique classical solution of (\ref{eqn:kinemod}), which exists globally in time:
$$(u,v)\in  (C^0([0,\infty); W^{1,p_0}(\Omega) )\cap C^{2,1}(\overline{\Omega}\times(0,\infty)))^2.$$
 Moreover the solution is positive on $ \overline{\Omega}\times(0,\infty)$
and the mass conservation law holds:
	\begin{equation*}
	\int_{\Omega}u(x, t)\,dx=\int_{\Omega}u_0 \,dx
	\quad \text{for\ all}\ t \in (0,\infty).
	\end{equation*}	
\end{prop} 

 Let us introduce the auxiliary function: 
\begin{equation}
 \label{eqn:defw}
w(t) :=(1-\Delta_N)^{-1}u(t) \qquad \mbox{for all }t\in[0,\infty),
\end{equation}
where $\Delta_N$ represents the Laplacian with the homogeneous Neumann boundary condition. 
That is,  for each $t\geq 0$ the function $w(t)$ is the unique solution of the following boundary value problem to the elliptic equation:
\begin{equation}\label{w-eq}
\begin{cases}
-\Delta w+w=u \qquad&x \in \Omega,\\
\frac{\partial w}{\partial \nu} = 0 &x \in \partial\Omega.
\end{cases}
\end{equation} 
Here we remark $w = (1-\Delta_N)^{-1}u \in C^{2,1}(\overline{\Omega}\times(0,\infty))$ by the elliptic regularity theory.

\medskip
The uniform-in-time lower bounds for $v$ and $w$ are obtained in the following lemma. 
The bound for $w$ comes from the positivity of the Helmholtz operator and the mass conservation (see \cite[Corollary 2.3]{Anh19}). 
Similarly, the lower estimate for the heat kernel and the mass conservation imply the lower bound for $v$ (\cite[Lemma 3.1]{FS2018},  \cite[Lemma 3.2]{FJ21}). 
Especially as noted in  \cite[Lemma 3.1]{FS2018}, the lower bound $v_\ast$ is independent of $\tau>0$.
\begin{lemm}
\label{lemm:Minvw}
 Let $\tau>0$. Assume {\bf (A0)} and {\bf (A2)}.  
 Suppose that $(u,v)$ be a solution of \eqref{eqn:kinemod} and 
 $w$ be the function defined in \eqref{eqn:defw}. Then there exist positive constants $v_\ast= v_\ast(n,\Omega,u_0,v_0)$ and $w_\ast= w_\ast(n,\Omega,u_0,v_0)$ satisfying
\[
 v \geq v_\ast \ \mbox{ and } \ w \geq w_\ast \quad \mbox{ in } \overline{\Omega} \times [0,\infty),
\]  
where constants $v_\ast$ and $w_\ast$ are independent of $\tau>0$.
\end{lemm}

Let us introduce the auxiliary identity, which is the specific structure of \eqref{eqn:kinemod}.

\begin{lemm}
 Let $\tau>0$. Assume {\bf (A0)} and {\bf (A2)}.  
 \label{lemm:eqnw}
$u$, $v$ and $w$ satisfy the following identity:
\begin{equation}
\label{eqn:parabw}
 w_t + \gamma (v) u = (1-\Delta_N)^{-1} (\gamma (v) u) \quad \mbox{ in } \Omega \times (0,\infty)
\end{equation}
and also
\begin{eqnarray}\label{eqn:evolutionw}
 w_t + \gamma (v) (\Delta w-w) = (1-\Delta_N)^{-1} (\gamma (v) u) \quad \mbox{ in } \Omega \times (0,\infty),
\end{eqnarray}
\end{lemm}
\begin{proof}
By taking $(1-\Delta_N)^{-1}$ to the first equation of (\ref{eqn:kinemod}), 
we obtain \eqref{eqn:parabw}.  
\eqref{w-eq} and \eqref{eqn:parabw} imply \eqref{eqn:evolutionw}.
\end{proof}

The following lemmas play key roles in Section \ref{sec:refined comparison estimate}.

\begin{lemm}\label{calculation_v}
Let $\ell\in (0,1) \cup (1,\infty)$ and $(u,v)$ be a solution of \eqref{eqn:kinemod} in $\Omega \times (0,\infty)$. 
Then the following identity holds:
\begin{eqnarray*}
\left(
\tau \frac{\partial}{\partial t} - \Delta + 1\right)v^{1-\ell} 
 =  (1-\ell) \frac{u}{v^\ell} + \ell v^{1-\ell} + \frac{\ell(1-\ell)}{v^{1+\ell}} |\nabla v|^2
 \quad \mbox{ in } \Omega \times (0,\infty).
\end{eqnarray*}
\end{lemm}
\begin{proof}
The straightforward calculation leads us to 
\begin{eqnarray*}
 \tau \frac{\partial v^{1-\ell}}{\partial t} - \Delta v^{1-\ell} + v^{1-\ell}
& = & \tau \frac{(1-\ell)}{v^\ell} v_t -\frac{(1-\ell)}{v^\ell} \Delta v + v^{1-\ell} + \frac{\ell(1-\ell)}{v^{1+\ell}} |\nabla v|^2 \\
& = & \frac{(1-\ell)}{v^\ell} \left( \tau v_t - \Delta v + v \right) + \ell v^{1-\ell} + \frac{\ell(1-\ell)}{v^{1+\ell}} |\nabla v|^2.
\end{eqnarray*}
By inserting the second equation of \eqref{eqn:kinemod} into the above, we have 
\begin{eqnarray*}
 \tau \frac{\partial v^{1-\ell}}{\partial t} - \Delta v^{1-\ell} + v^{1-\ell}
 =  (1-\ell) \frac{u}{v^\ell} + \ell v^{1-\ell} + \frac{\ell(1-\ell)}{v^{1+\ell}} |\nabla v|^2,
\end{eqnarray*}
which is the desired identity.
\end{proof}

\begin{lemm}\label{calculation_w}
Let $\ell\in (0,1) \cup (1,\infty)$ and $(u,v)$ be a solution of \eqref{eqn:kinemod} in $\Omega \times (0,\infty)$. 
Then the auxiliary function $w$ defined in \eqref{eqn:defw} satisfies the following identity:
\begin{eqnarray*}
\left(- \Delta + 1\right) w^{1-\ell} 
 =(1-\ell) \frac{u}{w^\ell} + \ell w^{1-\ell} + \frac{\ell(1-\ell)}{w^{\ell+1}} |\nabla w|^2
 \quad \mbox{ in } \Omega \times (0,\infty).
\end{eqnarray*}
\end{lemm}
\begin{proof}
Proceeding the same calculations as in the proof of Lemma \ref{calculation_v}, we have
\begin{eqnarray*}
- \Delta w^{1-\ell} + w^{1-\ell}
& = & \frac{(1-\ell)}{w^\ell} \left( -\Delta w + w \right) +\ell w^{1-\ell}+ \frac{\ell(1-\ell)}{w^{\ell+1}} |\nabla w|^2,
\end{eqnarray*}
and thus due to \eqref{w-eq} it follows
\begin{eqnarray*}
- \Delta w^{1-\ell} + w^{1-\ell}
& = & (1-\ell) \frac{u}{w^\ell} + \ell w^{1-\ell} + \frac{\ell(1-\ell)}{w^{\ell+1}} |\nabla w|^2.
\end{eqnarray*}
We complete the proof.
\end{proof}

\section{Refined comparison estimates}
\label{sec:refined comparison estimate}
In the previous work \cite{FJ21}, the following estimate is shown:
$$
 v \leq C(w+1) \quad \mbox{ in } \Omega \times (0,\infty)
$$
with some $C>0$. 
In this section, we will refine the above estimate and show $v \simeq w$.

\begin{lemm}
 \label{lemm:Estvw}
Assume $\tau>0$ and $\gamma (v)=v^{-k}$ with $k>0$.  
Let $(u,v)$ be a solution of \eqref{eqn:kinemod} in $\Omega \times (0,\infty)$ and 
 $w$ be the function defined in \eqref{eqn:defw}. 
Then there exist positive constants $d_1$ and $d_2$ such that 
\[
 d_1 v \leq w \leq d_2 v \quad \mbox{ in } \Omega \times (0,\infty).
\]
\end{lemm}
\begin{proof}
By \cite[Lemma 7 and Remark 9]{FJ21}, we can find a constant $C>0$ satisfying 
\[
 v \leq C(w + 1) \quad \mbox{ in } \Omega \times (0,\infty),
\]
from which together with Lemma \ref{lemm:Minvw} it follows that 
\begin{equation}
 \label{eqn:Exisd1}
v \leq C(w +1) \leq C\left( 1 + \frac{1}{w_\ast} \right) w \quad \mbox{ in } \Omega \times (0,\infty).
\end{equation} 
Then we can choose the constant $d_1$ as $\left(C+\frac{C}{w_\ast}\right)^{-1}$.

Next, we shall focus on the constant $d_2$.  
We divide the proof into three cases: 
{\bf (i)} $k\in(0,1)$, {\bf (ii)} $ k=1$, {\bf (iii)} $k>1$.

\medskip
{\bf (i)} We assume $k \in (0,1)$.  By Lemma \ref{calculation_v} it follows
\begin{eqnarray*}
\left(
\tau \frac{\partial}{\partial t} - \Delta + 1\right)v^{1-k} 
 =  (1-k) \frac{u}{v^k} + kv^{1-k} + \frac{k(1-k)}{v^{1+k}} |\nabla v|^2,
\end{eqnarray*}
and by taking $\frac{\tau}{1-k}(1-\Delta_N)^{-1}$ to the above equation, we obtain that 
\begin{eqnarray}
\label{eqn:heateqnv1-k}
&& 
\nn
\left(
\tau \frac{\partial}{\partial t} - \Delta + 1\right)
 \left\{ \frac{\tau}{1-k} (1-\Delta_N)^{-1}v^{1-k} \right\} \\
& = & \tau (1-\Delta_N)^{-1} \left( \frac{u}{v^k} \right) 
+ \frac{\tau}{1-k} (1-\Delta_N)^{-1} \left\{ kv^{1-k} + \frac{(1-k)k}{v^{1+k}} |\nabla v|^2 \right\}.
\end{eqnarray}
On the other hand, it follows from \eqref{w-eq} and Lemma \ref{lemm:eqnw} that 
\begin{equation}
\label{eqn:heateqnw}
\left(\tau \frac{\partial}{\partial t} - \Delta +1 \right) w= u - \tau \frac{u}{v^k} + \tau (1-\Delta_N)^{-1} \left( \frac{u}{v^k} \right),
\end{equation}
which together with (\ref{eqn:heateqnv1-k}) yields that 
\begin{eqnarray*}
&& 
\left(\tau \frac{\partial}{\partial t} - \Delta +1 \right)
\left\{ w - \frac{\tau}{1-k} (1-\Delta_N)^{-1} v^{1-k} \right\}  \\
& = & u - \left\{ \tau \frac{u}{v^k} + \frac{\tau k}{1-k} (1-\Delta_N)^{-1} v^{1-k} +  (1-\Delta_N)^{-1} \left( \frac{\tau k}{v^{1+k}}|\nabla v|^2 \right) \right\}.
\end{eqnarray*}
In light of the nonnegativity of $u$ and $v$ and the maximum principle, it follows
\begin{eqnarray*}
\left(\tau \frac{\partial}{\partial t} - \Delta +1 \right)
\left\{ w - \frac{\tau}{1-k} (1-\Delta_N)^{-1} v^{1-k} \right\}  \leq u= \left(\tau \frac{\partial}{\partial t} - \Delta +1 \right)v.
\end{eqnarray*}
Moreover, the function $\left\{ w - \frac{\tau}{1-k} (1-\Delta_N)^{-1} v^{1-k} \right\}$ satisfies the Neumann boundary condition and by the maximum principle, we see
\begin{eqnarray*}
\left\{ w - \frac{\tau}{1-k} (1-\Delta_N)^{-1} v^{1-k} \right\} \bigg|_{t=0}
& = & (1-\Delta_N)^{-1} u_0 - \frac{\tau}{1-k} (1-\Delta_N)^{-1} v_0^{1-k} \\
& \leq & \|u_0\|_{L^\infty} \\
&\leq & \left(1+ \frac{\|u_0\|_{L^\infty}}{v_\ast} \right)v_0. 
\end{eqnarray*}
Therefore we have
\begin{align*}
\begin{cases}
&\left(\tau \frac{\partial}{\partial t} - \Delta +1 \right)
\left\{ w - \frac{\tau}{1-k} (1-\Delta_N)^{-1} v^{1-k} \right\}  
\leq 
\left( \tau \frac{\partial}{\partial t} - \Delta +1 \right)  \left(1+ \frac{\|u_0\|_{L^\infty}}{v_\ast} \right)v,\\[5mm]
&\left\{ w - \frac{\tau}{1-k} (1-\Delta_N)^{-1} v^{1-k} \right\} \bigg|_{t=0}
\leq 
 \left(1+ \frac{\|u_0\|_{L^\infty}}{v_\ast} \right)v\bigg|_{t=0},\\[5mm]
& \frac{\partial }{\partial \nu } \left\{ w - \frac{\tau}{1-k} (1-\Delta_N)^{-1} v^{1-k} \right\}  
 = \frac{\partial }{\partial \nu }  \left(1+ \frac{\|u_0\|_{L^\infty}}{v_\ast} \right)v =0,
 \end{cases}
\end{align*}
and we can employ the comparison principle to obtain  
\[
 w - \frac{\tau}{1-k} (1-\Delta_N)^{-1} v^{1-k}
\leq \left(1+ \frac{\|u_0\|_{L^\infty}}{v_\ast} \right) v,
\]
that is, 
\begin{equation}
\label{eqn:est1w}
 w 
\leq \left(1+ \frac{\|u_0\|_{L^\infty}}{v_\ast} \right) v + \frac{\tau}{1-k} (1-\Delta_N)^{-1} v^{1-k}.
\end{equation}
To complete the proof, we need to estimate the second term of the right-hand side in the above. 
By Lemma \ref{calculation_w}, it follows
\begin{eqnarray*}
(- \Delta +1) w^{1-k}
 = 
(1-k) \frac{u}{w^k} + kw^{1-k} + \frac{k(1-k)}{w^{k+1}} |\nabla w|^2,
\end{eqnarray*}
which together with (\ref{eqn:Exisd1}) implies
\begin{eqnarray*}
(- \Delta +1) w^{1-k} 
\geq   kd_1^{1-k }v^{1-k}.
\end{eqnarray*}
By the comparison theorem, we imply that 
\[
w^{1-k} \geq kd_1^{1-k} (1-\Delta_N)^{-1} v^{1-k}.
\]
We deduce from the above and (\ref{eqn:est1w}) that 
\begin{eqnarray}\label{eqn:est11w}
\nn
w  & \leq&
 \left(1+ \frac{\|u_0\|_{L^\infty}}{v_\ast} \right) v + \frac{\tau}{(1-k)kd_1^{1-k}} w^{1-k}\\
&=& \left(1+ \frac{\|u_0\|_{L^\infty}}{v_\ast} \right) v + \frac{\tau}{(1-k)kd_1^{1-k}} \cdot \frac{1}{w^k} w. 
\end{eqnarray}
If 
$$
 \frac{\tau}{(1-k)kd_1^{1-k}} \cdot \frac{1}{w^k} \leq \frac{1}{2},
 \quad \mbox{that is}, \quad
w^k \geq \frac{2\tau}{(1-k)kd_1^{1-k}},$$
then \eqref{eqn:est11w} implies that 
\[
w \leq 2 \left(1+ \frac{\|u_0\|_{L^\infty}}{v_\ast} \right)v;
\]
Otherwise if 
$$w^k \leq \frac{2\tau}{(1-k)kd_1^{1-k}},$$
it follows from \eqref{eqn:est11w} that
\[
w  \leq  \left( 1+ \frac{\|u_0\|_{L^\infty}}{v_\ast} \right) v 
+ \frac{\tau}{(1-k)kd_1^{1-k}}\left(\frac{2\tau}{k(1-k)d_1^{1-k}}\right)^{(1-k)/k}.
\]
Therefore combining these estimates and using Lemma \ref{lemm:Minvw} we obtain that 
\begin{eqnarray*}
w & \leq & 2 \left( 1+ \frac{\|u_0\|_{L^\infty}}{v_\ast} \right) v + \frac{\tau}{(1-k)kd_1^{1-k}}\left(\frac{2\tau}{k(1-k)d_1^{1-k}}\right)^{(1-k)/k} \\
& \leq & 2 \left( 1+ \frac{\|u_0\|_{L^\infty}}{v_\ast} \right) v + \frac{\tau}{(1-k)kd_1^{1-k}v_\ast}\left(\frac{2\tau}{k(1-k)d_1^{1-k}}\right)^{(1-k)/k} v,
\end{eqnarray*}
which guarantees the existence of the constant $d_2$ in the case where $k \in (0,1)$.

\medskip
{\bf (ii)} We assume $k=1$.  By the direct calculations, it follows
\begin{eqnarray*}
\left(\tau \frac{\partial}{\partial t} - \Delta +1 \right) \log v 
& = & \tau \frac{v_t}{v} - \frac{1}{v} \Delta v
+ \frac{1}{v^2} |\nabla v|^2 + \log v \\
& = & \frac{1}{v} \left( \tau v_t - \Delta v + v \right) - \frac{v}{v} + \frac{1}{v^2}|\nabla v|^2 + \log v \\
& = & \frac{u}{v} -1 + \frac{1}{v^2} |\nabla v|^2 + \log v.
\end{eqnarray*}
Putting $A:=1-\log v_\ast$, we can confirm that
\begin{eqnarray*}
\left(\tau \frac{\partial}{\partial t} - \Delta +1 \right) \left( A + \log v \right)
& = & A +\frac{u}{v}  -1  + \frac{1}{v^2} |\nabla v|^2 + \log v \\
&\geq& \frac{u}{v} + \frac{1}{v^2} |\nabla v|^2
\end{eqnarray*}
since Lemme \ref{lemm:Minvw} implies $A-1+\log v \geq 0$. 
Hence by taking $\tau(1-\Delta_N)^{-1} $ to the above, it follows
$$
\left(\tau \frac{\partial}{\partial t} - \Delta +1 \right)
\left\{\tau  (1-\Delta_N)^{-1} (A+\log v) \right\}
\geq \tau (1-\Delta_N)^{-1} \left( \frac{u}{v} + \frac{1}{v^2} |\nabla v|^2\right).
$$
Combining the above with \eqref{eqn:heateqnw} we obtain
\begin{eqnarray*}
&&\left(\tau \frac{\partial}{\partial t} - \Delta +1 \right)
\left\{ w- \tau(1-\Delta_N)^{-1} (A+\log v) \right\}\\
& \leq & \left\{ u - \tau \frac{u}{v} + \tau (1-\Delta_N)^{-1} \frac{u}{v} \right\} 
- \tau (1-\Delta_N)^{-1} \left( \frac{u}{v} + \frac{1}{v^2} |\nabla v|^2\right) \\
& = & u - \tau \left\{ \frac{u}{v} + (1-\Delta_N)^{-1} \left( \frac{1}{v^2} |\nabla v|^2 \right) \right\},
\end{eqnarray*}
and by the nonegativity of $u$ and $v$ and the comparison principle, it follows
\begin{eqnarray*}
\left(\tau \frac{\partial}{\partial t} - \Delta +1 \right)
\left\{ w- \tau(1-\Delta_N)^{-1} (A+\log v) \right\}
\leq u.
\end{eqnarray*} 
We also employ the comparison principle to see
\begin{eqnarray*}
\left\{ w- \tau(1-\Delta_N)^{-1} (A+\log v) \right\}\bigg|_{t=0}
& \leq & (1-\Delta_N)^{-1} u_0 \\
& \leq & \|u_0\|_{L^\infty} \\
&\leq& \left( 1+ \frac{\|u_0\|_{L^\infty}}{v_\ast} \right) v_0.
\end{eqnarray*}
Therefore we have
\begin{align*}
\begin{cases}
&\left(\tau \frac{\partial}{\partial t} - \Delta +1 \right)
\left\{ w- \tau(1-\Delta_N)^{-1} (A+\log v) \right\}
\leq 
\left( \tau \frac{\partial}{\partial t} - \Delta +1 \right)  \left(1+ \frac{\|u_0\|_{L^\infty}}{v_\ast} \right)v,\\[5mm]
&\left\{ w- \tau(1-\Delta_N)^{-1} (A+\log v) \right\}\bigg|_{t=0}
\leq 
 \left(1+ \frac{\|u_0\|_{L^\infty}}{v_\ast} \right)v\bigg|_{t=0}, \\[5mm]
& \frac{\partial }{\partial \nu } 
\left\{ w- \tau(1-\Delta_N)^{-1} (A+\log v) \right\}
 = \frac{\partial }{\partial \nu }  \left(1+ \frac{\|u_0\|_{L^\infty}}{v_\ast} \right)v =0,
 \end{cases}
\end{align*}
and then by the comparison principle it follows
\begin{equation}
\label{eqn:est2w}
w  \leq \left( 1+ \frac{\|u_0\|_{L^\infty}}{v_\ast} \right) v +  \tau (1-\Delta_N)^{-1} \left( 1 + \log \frac{v}{v_\ast} \right).
\end{equation}
We focus on the second term of the right-hand side in the above. 
By the direct calculations we see that 
\begin{eqnarray}\label{w_log}
\nn
(-\Delta +1)\log w 
& = & - \frac{1}{w} \Delta w + \frac{1}{w^2} |\nabla w|^2 + \log w \\
\nn
& = & \frac{1}{w} \left( - \Delta w + w \right) -1 + \frac{1}{w^2} |\nabla w|^2 + \log w \\
& = & \frac{u}{w} -1 + \log w + \frac{1}{w^2} |\nabla w|^2,
\end{eqnarray}
and by \eqref{eqn:Exisd1} it follows
\begin{eqnarray*}
(-\Delta +1) ( \log w + 2- \log (d_1v_\ast)) & = & \frac{u}{w} + \frac{1}{w^2}|\nabla w|^2 +1+  \log \frac{w}{d_1v_\ast} \\
&\geq & 1+ \log \frac{d_1 v}{d_1v_\ast} \\
&= &1+ \log \frac{v}{v_\ast} .
\end{eqnarray*}
Hence by the comparison principle we have
\[
\log w + 2- \log (d_1v_\ast)  \geq (1-\Delta_N)^{-1} \left(1+ \log \frac{v}{v_\ast}\right),
\]
which together with (\ref{eqn:est2w}) implies that 
\[
w  \leq  \left( 1 + \frac{\|u_0\|_{L^\infty}}{v_\ast}  \right)v 
+\tau (\log w + 2- \log (d_1v_\ast)).
\]
Since 
\[
\frac{s}{2} - \tau \log s \geq \tau -\tau  \log (2\tau) \quad \mbox{ for } s>0,
\]
we have
\[
\frac{w}{2}  \leq   \left( 1 + \frac{\|u_0\|_{L^\infty}}{v_\ast}  \right) v 
+ \tau (1+  \log (2\tau)- \log (d_1v_\ast))
\]
and by Lemma \ref{lemm:Minvw} it follows 
\begin{eqnarray*}
w \leq   2\left( 1 + \frac{\|u_0\|_{L^\infty}}{v_\ast}  \right) v 
+  \frac{2\tau|1+  \log (2\tau)- \log (d_1v_\ast)|}{v_\ast} v.
\end{eqnarray*}
Then we find the constant $d_2$ in the case where $k=1$. 

\medskip
{\bf (iii)} We assume $k>1$. By Lemma \ref{calculation_v} it follows
\begin{eqnarray*}
 \left(\tau \frac{\partial}{\partial t} - \Delta +1 \right) v^{1-k} 
& =&   (1-k) \frac{u}{v^k} + \frac{k}{v^{k-1}} + \frac{(1-k)k}{v^{k+1}} |\nabla v|^2\\
& \leq&   -(k-1) \frac{u}{v^k} + \frac{k}{v^{k-1}},
\end{eqnarray*} 
and then by taking $\frac{\tau}{k-1} (1-\Delta_N)^{-1}$ we have
\begin{eqnarray}\label{largekineq}
\nn
&& \left(\tau \frac{\partial}{\partial t} - \Delta +1 \right) 
 \left( \frac{\tau}{k-1} (1-\Delta_N)^{-1} v^{1-k} \right)\\
& \leq &
-\tau (1-\Delta_N)^{-1} \left( \frac{u}{v^k} \right) 
 + \frac{\tau k}{k-1} (1-\Delta_N)^{-1} \left( \frac{1}{v^{k-1}} \right).
\end{eqnarray}
By \eqref{eqn:heateqnw} and \eqref{largekineq}, it follows
\begin{eqnarray*}
 &&  \left(\tau \frac{\partial}{\partial t} - \Delta +1 \right) 
  \left\{ w + \frac{\tau}{k-1} (1-\Delta_N)^{-1} \left( \frac{1}{v^{k-1}} \right) \right\} \\
&\leq & \left( u -\tau \frac{u}{v^k} +\tau(1-\Delta_N)^{-1} \left( \frac{u}{v^k} \right) \right)
-\tau (1-\Delta_N)^{-1} \left( \frac{u}{v^k} \right) + \frac{\tau k}{k-1} (1-\Delta_N)^{-1} \left(  \frac{1}{v^{k-1}} \right) \\
& = & u -\tau \frac{u}{v^k} + \frac{\tau k}{k-1} (1-\Delta_N)^{-1} \left( \frac{1}{v^{k-1}} \right).
\end{eqnarray*}
Taking 
$$B :=\frac{\tau k}{k-1}(1-\Delta_N)^{-1} \left( \frac{1}{v^{k-1}_\ast} \right) 
= \frac{\tau k}{k-1} \cdot \frac{1}{v_\ast^{k-1}},$$
we see that 
\begin{eqnarray*}
 &&  \left(\tau \frac{\partial}{\partial t} - \Delta +1 \right) 
  \left\{ w + \frac{\tau}{k-1} (1-\Delta_N)^{-1} \left( \frac{1}{v^{k-1}} \right) -B\right\} \\
&\leq & u -\tau \frac{u}{v^k} + \frac{\tau k}{k-1} (1-\Delta_N)^{-1} \left( \frac{1}{v^{k-1}} \right) -B \\
& = &  u -\tau \frac{u}{v^k} - \frac{\tau k}{k-1} (1-\Delta_N)^{-1} \left(\frac{1}{v^{k-1}_\ast} - \frac{1}{v^{k-1}} \right)  \\
& \leq & u,
\end{eqnarray*}
where we employed Lemma \ref{lemm:Minvw} and the comparison principle to see
$$
(1-\Delta_N)^{-1} \left(\frac{1}{v^{k-1}_\ast} - \frac{1}{v^{k-1}} \right)  \geq 0.
$$
Similarly by the comparison principle it follows
\begin{eqnarray*}
&&  \left\{ w + \frac{\tau}{k-1} (1-\Delta_N)^{-1} \frac{1}{v^{k-1}} -B\right\}   \bigg|_{t=0}\\
 & = & (1-\Delta)^{-1}u_0
+ \frac{\tau }{k-1}(1-\Delta_N)^{-1} \left( \frac{1}{v_0^{k-1}} - \frac{k}{v_\ast^{k-1}} 
\right) \\
& \leq & (1-\Delta_N)^{-1}u_0\\
& \leq& \|u_0\|_{L^\infty}\\
& \leq& \left( 1+ \frac{\|u_0\|_{L^\infty}}{v_\ast}\right) v_0.  
\end{eqnarray*}
Therefore we have
\begin{align*}
\begin{cases}
&\left(\tau \frac{\partial}{\partial t} - \Delta +1 \right)
 \left\{ w + \frac{\tau}{k-1} (1-\Delta_N)^{-1} \left( \frac{1}{v^{k-1}} \right) -B\right\}
\leq 
\left( \tau \frac{\partial}{\partial t} - \Delta +1 \right)  \left(1+ \frac{\|u_0\|_{L^\infty}}{v_\ast} \right)v,\\[5mm]
& \left\{ w + \frac{\tau}{k-1} (1-\Delta_N)^{-1} \left( \frac{1}{v^{k-1}} \right) -B\right\}  \bigg|_{t=0}
\leq 
 \left(1+ \frac{\|u_0\|_{L^\infty}}{v_\ast} \right)v\bigg|_{t=0}, \\[5mm]
& \frac{\partial }{\partial \nu } 
 \left\{ w + \frac{\tau}{k-1} (1-\Delta_N)^{-1} \left( \frac{1}{v^{k-1}} \right) -B\right\}
 = \frac{\partial }{\partial \nu }  \left(1+ \frac{\|u_0\|_{L^\infty}}{v_\ast} \right)v =0,
 \end{cases}
\end{align*}
and the comparison theorem implies that 
\[
w - \frac{\tau k}{(k-1)v_\ast^{k-1}} + \frac{\tau}{k-1} (1-\Delta_N)^{-1}  \left( \frac{1}{v^{k-1}} \right)
 \leq \left(1+ \frac{\|u_0\|_{L^\infty}}{v_\ast} \right) v 
\]
and then by Lemma \ref{lemm:Minvw} it follows
\begin{eqnarray*}
w 
  & \leq & 
 \left(1+ \frac{\|u_0\|_{L^\infty}}{v_\ast} \right)  v 
 + \frac{\tau k}{(k-1)v^{k-1}_\ast}  \\
& \leq & \left( 1+ \frac{\|u_0\|_{L^\infty}}{v_\ast} +\frac{\tau k}{(k-1)v^{k}_\ast}  \right)v,
\end{eqnarray*}
which guarantees the constant $d_2$ in the case where $k>1$.
Therefore the proof is complete. 
\end{proof}

%\cred{%
\begin{rema}\label{remA-1}
Lemma \ref{lemm:Estvw} also holds true when we assume {\bf (A1)$^\prime$}. 
Here we focus on the case $k\in (0,1)$.  
Instead of \eqref{eqn:heateqnw}, we have 
\begin{equation*}
\left(\tau \frac{\partial}{\partial t} - \Delta +1 \right) w
= u - \tau u \gamma (v) + \tau (1-\Delta_N)^{-1} \left( u \gamma (v) \right).
\end{equation*}
Since the maximum principle and {\bf (A1)$^\prime$} imply
$$
(1-\Delta_N)^{-1} \left( u \gamma (v) \right) 
\leq \Gamma_{max} (1-\Delta_N)^{-1} \left( \frac{u}{v^k} \right),
$$
it follows
$$
\left(\tau \frac{\partial}{\partial t} - \Delta +1 \right) w
= u - \tau u \gamma (v) 
+  \Gamma_{max} \tau (1-\Delta_N)^{-1} \left( \frac{u}{v^k} \right).
$$
Multiplying \eqref{eqn:heateqnv1-k} by $\Gamma_{max}$ and combining this with the above, we have
\begin{eqnarray*}
\left(\tau \frac{\partial}{\partial t} - \Delta +1 \right)
\left\{ w - \Gamma_{max} \frac{\tau}{1-k} (1-\Delta_N)^{-1} v^{1-k} \right\}  \leq u.
\end{eqnarray*}
Proceeding the same procedure, we can prove the existence of $d_2$.
\end{rema}
%}%

Thanks to the previous lemma, we can derive upper estimates for $(1-\Delta_N)^{-1} (u\gamma (v))$.
\begin{lemm}
 \label{lemm:bounduvk}
Assume $\tau>0$ and $\gamma (v)=v^{-k}$ with $k>0$.
Let $d_1$ and $d_2$ be the constants in Lemma \ref{lemm:Estvw}.  
The following estimates hold:
\begin{enumerate}
\item[{\bf (i)}] If $k \in (0,1)$, 
\[
 (1-\Delta_N)^{-1} \frac{u}{v^k} \leq \frac{d_2^k}{1-k} w^{1-k}.
\]
\item[{\bf (ii)}]  If $k=1$,
\[
 (1-\Delta_N)^{-1} \frac{u}{v} \leq d_2 \left( 1 + \log \frac{w}{w_\ast} \right).
\]
\item[{\bf (iii)}]  If $k>1$,
\[
(1-\Delta_N)^{-1} \frac{u}{v^k} \leq  \frac{d_2^k}{k-1} \left( \frac{k}{w_\ast^{k-1}} - \frac{1}{w^{k-1}}\right).  
\]
\end{enumerate}
\end{lemm}
\begin{proof} 
We first assume $k \in (0,1)$. 
It follows from Lemma \ref{calculation_w} and Lemma \ref{lemm:Estvw} that 
\begin{eqnarray*}
\left(-\Delta +1\right)w^{1-k}
& = & (1-k) \frac{u}{w^k} + kw^{1-k} + \frac{k(1-k)}{w^{k+1}}|\nabla w|^2 \\
& \geq & \frac{(1-k)}{d_2^k} \frac{u}{v^k},  
\end{eqnarray*} 
then by the comparison principle we see that
\[
 w^{1-k} \geq \frac{1-k}{d_2^k} (1-\Delta_N)^{-1} \frac{u}{v^k}.
\]
Then {\bf (i)} holds. 

Next, we consider the case $k=1$.
By \eqref{w_log}, Lemma \ref{lemm:Minvw} and Lemma \ref{lemm:Estvw} it follows
\begin{eqnarray*}
\left(-\Delta +1\right)(\log w +1-\log w_\ast) 
& = & \frac{u}{w}   + \log \frac{w}{w_\ast} + \frac{1}{w^2} |\nabla w|^2  \\
& \geq & \frac{u}{d_2v},
\end{eqnarray*}
thus we have 
\[
 \log w + 1 - \log w_\ast \geq \frac{1}{d_2} (1-\Delta_N)^{-1} \frac{u}{v},
\]
which concludes {\bf (ii)}.

Finally, we assume $k>1$.  From Lemma \ref{calculation_w} we have
\begin{eqnarray*}
\left(-\Delta +1\right)w^{1-k}
& = & (1-k) \frac{u}{w^k} + kw^{1-k} + \frac{k(1-k)}{w^{k+1}} |\nabla w|^2\\
&\leq &-\frac{(k-1)}{d_2^k} \frac{u}{v^k}+ kw^{1-k}.
\end{eqnarray*}
Then we deduce from the above inequality, Lemma \ref{lemm:Minvw} and Lemma \ref{lemm:Estvw} that 
\begin{eqnarray*}
\left(-\Delta +1\right) \left( \frac{k}{w^{k-1}_\ast} - \frac{1}{w^{k-1}}\right)
& \geq & \frac{(k-1)}{d_2^k} \frac{u}{v^k} + \frac{k}{w^{k-1}_\ast} - \frac{k}{w^{k-1}}\\
&\geq&\frac{(k-1)}{d_2^k} \frac{u}{v^k}.
\end{eqnarray*}
Taking $(1-\Delta_N)^{-1} $ to the above, we complete the proof of {\bf (iii)}.
\end{proof}
%
%
%\cred{%
\begin{rema}\label{remA-11}
Lemma \ref{lemm:bounduvk} holds true when we assume {\bf (A1)$^\prime$} 
since the above proof is based on Lemma \ref{lemm:Estvw}, which also holds true under {\bf (A1)$^\prime$} as noted in Remark \ref{remA-1}. 
\end{rema}
%}%
%
%
%
\section{Boundedness of $v$}
Let $(u,v)$ be a solution of (\ref{eqn:kinemod}) in $\Omega \times (0,\infty)$ and let $w$ be the function defined in (\ref{eqn:defw}).  We focus on the evolution equation \eqref{eqn:evolutionw}:
\begin{eqnarray*}
 w_t + \gamma (v) (\Delta w-w) = (1-\Delta_N)^{-1} (\gamma (v) u) \quad \mbox{ in } \Omega \times (0,\infty).
\end{eqnarray*}
Invoking the estimates for nonlinear terms established in Lemma \ref{lemm:bounduvk}, 
we will show the boundedness of $v$ and $w$ in this section. 
First of all, we shall show the boundedness of $\|v\|_{L^p(\Omega)}$ and $\|w\|_{L^p(\Omega)}$ for any $p\in[1,\infty)$.
\begin{lemm}
 \label{lemm:boundvLp}
Let $\tau>0$.  Assume {\bf (A1)} and {\bf (A2)}. 
For any $p \geq 1$, $L^p(\Omega)$ norms of the functions $v$ and $w$ are uniformly bounded on $[0,\infty)$, i.e.,
\[
 \sup_{t \geq 0} \|v(\cdot,t)\|_{L^p} < \infty
 \quad\mbox{and}\quad
   \sup_{t \geq 0} \|w(\cdot,t)\|_{L^p} < \infty.
\]
\end{lemm}

To prove Lemma \ref{lemm:boundvLp} we prepare three lemmas.

\begin{lemm}\label{lem:preLp}
Let $\tau>0$.  Assume {\bf (A1)} and {\bf (A2)}.
For any $p>k+1$ the following inequality holds:
 \begin{equation*}
 \frac{d}{dt} \int_\Omega w^p\, dx + \frac{4p(p-k-1) d_1^k}{(p-k)^2} 
\int_\Omega \left( |\nabla w^{\frac{p-k}{2}}|^2 + w^{p-k}\right)\, dx
   \leq p\int_\Omega \left\{ (1-\Delta_N)^{-1} \frac{u}{v^k} \right\} w^{p-1} \, dx.
\end{equation*}
\end{lemm}
\begin{proof}
Multiplying (\ref{eqn:parabw}) by $w^{p-1}$ and integrating over $\Omega$, we have that 
\begin{eqnarray}\label{directcal}
 \int_\Omega w_t w^{p-1} \,dx + \int_\Omega \frac{u}{v^k} w^{p-1} \,dx & = & 
\int_\Omega \left\{(1-\Delta_N)^{-1} \frac{u}{v^k}\right\} w^{p-1}\,dx.
\end{eqnarray}
In view of Lemma \ref{lemm:Estvw}, \eqref{w-eq} and integration by parts, it follows
\begin{eqnarray*}
\nn
 \int_\Omega \frac{u}{v^k} w^{p-1}\, dx 
 & \geq & d_1^k \int_\Omega \frac{u}{w^k} w^{p-1}\, dx \\
 \nn
& = & d_1^k \int_\Omega (-\Delta w + w) w^{p-k-1}\, dx \\
\nn
& = & \frac{4(p-k-1) d_1^k}{(p-k)^2} \int_\Omega |\nabla w^{\frac{p-k}{2}}|^2\, dx 
+ d_1^k \int_\Omega w^{p-k}\, dx \\
& \geq & \frac{4(p-k-1) d_1^k}{(p-k)^2} 
\int_\Omega \left( |\nabla w^{\frac{p-k}{2}}|^2 + w^{p-k}\right)\, dx,
\end{eqnarray*}
where we used
$$
 \frac{4(p-k-1)}{(p-k)^2} \leq 1 
 \quad\Longleftrightarrow \quad 
 (p-k-2)^2\geq 0.
$$
Combining the above with \eqref{directcal}, we complete the proof.
\end{proof}

%
%\cred{%
\begin{rema}\label{remA-2}
The similar result to Lemma \ref{lem:preLp} holds true when we assume {\bf (A1)$^\prime$} instead of {\bf (A1)}. 
When we assume {\bf (A1)} it follows 
\begin{eqnarray*}
 \int_\Omega w_t w^{p-1} \,dx + \int_\Omega u \gamma(v) w^{p-1} \,dx & \leq & 
\Gamma_{max} \int_\Omega \left\{(1-\Delta_N)^{-1} \frac{u}{v^k}\right\} w^{p-1}\,dx.
\end{eqnarray*}
instead of \eqref{directcal}.  
Since 
$$
\int_\Omega u \gamma(v) w^{p-1} \,dx  \geq \Gamma_{min} \io \frac{u}{v_k}w^{p-1},
$$
by proceeding the same lines we have
 \begin{equation*}
 \frac{d}{dt} \int_\Omega w^p\, dx + \Gamma_{min} \frac{4p(p-k-1) d_1^k}{(p-k)^2} 
\int_\Omega \left( |\nabla w^{\frac{p-k}{2}}|^2 + w^{p-k}\right)\, dx
   \leq \Gamma_{max} p\int_\Omega \left\{ (1-\Delta_N)^{-1} \frac{u}{v^k} \right\} w^{p-1} \, dx.
\end{equation*}
for $p>k+1$.
\end{rema}
%}%

Here we recall the Sobolev inequality:
\begin{equation}
\label{eqn:sobevI}
 \|f\|_{L^{\frac{2n}{n-2}}}^2 \leq K_{Sob} \left( \|\nabla f\|_{L^2}^2 + \|f\|_{L^2}^2 \right)\qquad f \in H^1(\Omega)
\end{equation}
with some constant $K_{Sob}>0$.

\begin{lemm}\label{lem:auxLp}
Let $\ell$ be a positive constant satisfying $\ell >\frac{n-2}{2}k-1$.  
The following inequality holds:
 \begin{eqnarray*}
\frac{d}{dt} \int_\Omega w^{\ell+k+1}\, dx 
+ \frac{2\ell(\ell+k+1) d_1^k}{(\ell+1)^2K_{Sob}} 
|\Omega|^{-\frac{\ell+1}{\ell+k+1}+\frac{n-2}{n}} 
\left( \int_\Omega w^{\ell+k+1}\, dx \right)^{\frac{\ell+1}{\ell+k+1}}  
\leq {\bf I},
\end{eqnarray*}
where
\begin{eqnarray*}
{\bf I}:= (\ell+k+1) \int_\Omega \left\{(1-\Delta_N)^{-1} \frac{u}{v^k}\right\} w^{\ell+k}dx
-\frac{2\ell(\ell+k+1) d_1^k}{(\ell+1)^2K_{Sob}} 
\left( \io w^{(\ell+1)\frac{n}{n-2}}\,dx \right)^\frac{n-2}{n}.
\end{eqnarray*}
\end{lemm}
\begin{proof}
It follows from the Sobolev inequality \eqref{eqn:sobevI} that
\begin{eqnarray}\label{est_diffusion2}
\left( \io w^{(\ell+1)\frac{n}{n-2}}\,dx \right)^\frac{n-2}{n}
\leq K_{Sob} \io  \left( |\nabla w^{\frac{\ell+1}{2}}|^2 + w^{\ell+1}\right)\, dx.
\end{eqnarray}
Combining Lemma \ref{lem:preLp} with $p=\ell+k+1 > k+1$ and \eqref{est_diffusion2}, 
we have
 \begin{eqnarray}
 \label{eqn:intwl+kI}
&&\frac{d}{dt} \int_\Omega w^{\ell+k+1}\, dx + \frac{4\ell(\ell+k+1) d_1^k}{(\ell+1)^2K_{Sob}} 
\left( \io w^{(\ell+1)\frac{n}{n-2}}\,dx \right)^\frac{n-2}{n} \nonumber \\
&& \qquad\qquad \leq (\ell+k+1) \int_\Omega \left\{(1-\Delta_N)^{-1} \frac{u}{v^k}\right\} w^{\ell+k}dx.
\end{eqnarray}
Since $\ell$ satisfies
$$\ell + k+1 < (\ell +1)\frac{n}{n-2},$$
it follows from the H\"older inequality that
\begin{equation}
\label{eqn:HolderIV}
\left( \int_\Omega w^{\ell+k+1} dx \right)^{\frac{\ell+1}{\ell+k+1}} \leq 
 |\Omega|^{\frac{\ell+1}{\ell+k+1}-\frac{n-2}{n}} 
\left(\int_\Omega w^{(\ell+1) \frac{n}{n-2}} dx \right)^{\frac{n-2}{n}}.
\end{equation}
Hence by \eqref{eqn:intwl+kI} and \eqref{eqn:HolderIV} we have
 \begin{eqnarray*}
&&\frac{d}{dt} \int_\Omega w^{\ell+k+1}\, dx 
+ \frac{2\ell(\ell+k+1) d_1^k}{(\ell+1)^2K_{Sob}} 
 |\Omega|^{-\frac{\ell+1}{\ell+k+1}+\frac{n-2}{n}} 
\left( \int_\Omega w^{\ell+k+1}\, dx \right)^{\frac{\ell+1}{\ell+k+1}}  \\
&& 
+ \frac{2\ell(\ell+k+1) d_1^k}{(\ell+1)^2K_{Sob}} 
\left( \io w^{(\ell+1)\frac{n}{n-2}}\,dx \right)^\frac{n-2}{n}
\leq (\ell+k+1) \int_\Omega \left\{(1-\Delta_N)^{-1} \frac{u}{v^k}\right\} w^{\ell+k}\,dx,
\end{eqnarray*}
which is the desired inequality.
\end{proof}

The condition $k \in (0,\frac{n}{n-2})$ comes from the next lemma, which is based on the regularity estimate in \cite{BS}. 
\begin{lemm}\label{lem:key}
Let $\ell$ be a positive constant satisfying $\ell >\frac{n-2}{2}k-1$.  
For any $q \in [1,\frac{n}{n-2})$ and $\varepsilon>0$, 
there exists some $C>0$ such that
\begin{eqnarray*}
\io w^{\ell+q}\,dx \leq \varepsilon \left( \io w^{(\ell+1)\frac{n}{n-2}}\,dx \right)^\frac{n-2}{n} +C.
\end{eqnarray*}
\end{lemm}
\begin{proof}
For fixed $q \in [1,\frac{n}{n-2})$ we can choose $r>1$ such that
$$
q< r<\frac{n}{n-2}
$$
and by the H\"{o}lder inequality it follows
\begin{eqnarray}\label{H1}
\int_\Omega w^{\ell+q}\,dx
\leq |\Omega|^{\frac{r-q}{\ell+r}}
 \left( \io w^{\ell+r} \right)^{\frac{\ell+q}{\ell+r}}.
\end{eqnarray}
On the other hand, the H\"{o}lder inequality implies that
\begin{eqnarray}\label{H2}
\nn
\int_\Omega w^{\ell+r}\, dx 
& = & 
\int_\Omega w^{(\ell+1)+(r-1)}\, dx \\
& \leq & 
\left(\int_\Omega  w^{(\ell+1)\frac{n}{n-2}}\,dx \right)^{\frac{n-2}{n}}
\left( \int_\Omega w^{(r-1)\frac{n}{2}}\, dx \right)^{\frac{2}{n}}.
\end{eqnarray} 
In light of the mass conservation law, the regularity estimate (\cite{BS}) guarantees that
$$
\io w^{(r-1)\frac{n}{2}}\, dx \leq C
$$
with some $C>0$ since
$$
(r-1)\frac{n}{2} < \frac{n}{n-2}.
$$
Combining \eqref{H1} and \eqref{H2} we have
\begin{eqnarray*}
\int_\Omega w^{\ell+q}\,dx
\leq
 |\Omega|^{\frac{r-q}{\ell+r}}
C^{\frac{2}{n} \cdot \frac{\ell+q}{\ell+r}} \left(\int_\Omega  w^{(\ell+1)\frac{n}{n-2}}\,dx \right)^{\frac{n-2}{n}\cdot \frac{\ell+q}{\ell+r}},
\end{eqnarray*}
and since $\frac{\ell+q}{\ell+r}<1$ we can invoke the Young inequality to complete the proof. 
\end{proof}

\begin{proof}[Proof of Lemma \ref{lemm:boundvLp}.]
We divide the proof into three cases: {\bf (i)} $k\in(0,1)$, {\bf (ii)} $ k=1$, {\bf (iii)} $k\in (1, \frac{n}{n-2})$.

\medskip
{\bf (i)} We assume $k \in (0,1)$.  
Let $\ell$ be a positive constant satisfying $\ell >\frac{n-2}{2}k-1$.  
In view of Lemma \ref{lemm:bounduvk} it follows 
\begin{eqnarray*}
(\ell+k+1) \int_\Omega \left\{(1-\Delta_N)^{-1} \frac{u}{v^k}\right\} w^{\ell+k}\,dx
\leq
\frac{d_2^k(\ell+k+1)}{(1-k)} \int_\Omega w^{\ell+1}\,dx,
\end{eqnarray*}
and then combining the above with Lemma \ref{lem:auxLp} we have
 \begin{eqnarray} \label{eqn:intwl+kII}
 \frac{d}{dt} \int_\Omega w^{\ell+k+1}\, dx 
+ \frac{2\ell(\ell+k+1) d_1^k}{(\ell+1)^2K_{Sob}} 
 |\Omega|^{-\frac{\ell+1}{\ell+k+1}+\frac{n-2}{n}} 
\left( \int_\Omega w^{\ell+k+1} dx \right)^{\frac{\ell+1}{\ell+k+1}}  
\leq {\bf I}_1
\end{eqnarray}
with 
\begin{eqnarray*}
{\bf I}_1 = \frac{d_2^k(\ell+k+1)}{(1-k)} \int_\Omega w^{\ell+1}\,dx
- \frac{2\ell(\ell+k+1) d_1^k}{(\ell+1)^2K_{Sob}} 
\left( \io w^{(\ell+1)\frac{n}{n-2}}\,dx \right)^\frac{n-2}{n}.
\end{eqnarray*}
Since $1<\frac{n}{n-2}$ we can apply Lemma \ref{lem:key} with sufficiently small $\varepsilon>0$ satisfying 
$$
\frac{d_2^k(\ell+k+1)}{(1-k)} \cdot \varepsilon
 \leq \frac{2\ell(\ell+k+1) d_1^k}{(\ell+1)^2K_{Sob}} 
$$
to have
$$
{\bf I}_1 \leq C_1
$$
with some $C_1>0$. 
Therefore it follows
 \begin{eqnarray*} 
 \frac{d}{dt} \int_\Omega w^{\ell+k+1}\, dx 
+ \frac{2\ell(\ell+k+1) d_1^k}{(\ell+1)^2K_{Sob}} 
|\Omega|^{-\frac{\ell+1}{\ell+k+1}+\frac{n-2}{n}} 
\left( \int_\Omega w^{\ell+k+1} dx \right)^{\frac{\ell+1}{\ell+k+1}}  
\leq C_1.
\end{eqnarray*}
By the ODE comparison principle, 
we deduce from the above inequality that for all $t\geq0$,
\begin{eqnarray*}
&&\int_\Omega w(x,t)^{\ell+k+1} dx \\
&& \leq \max \left\{ \int_\Omega w(x,0)^{\ell+k+1} dx, \ C_1^{\frac{\ell+k+1}{\ell+1}} 
\left(
\frac{2\ell(\ell+k+1) d_1^k}{(\ell+1)^2K_{Sob}} 
 |\Omega|^{-\frac{\ell+1}{\ell+k+1}+\frac{n-2}{n}} 
 \right)^{-\frac{\ell+k+1}{\ell+1}} \right\}.
\end{eqnarray*}
This claims that the standard $L^p$-norm of the function $w$ is uniformly bounded in time for any sufficiently large $p$, which together with the boundedness of the domain $\Omega$, we imply that for any $p \geq 1$ the $L^p$-norm of the function $w$ is uniformly bounded in time. We deduce from this and Lemma \ref{lemm:Estvw} that for any $p \geq 1$ the standard $L^p$-norm of the function $v$ is also uniformly bounded in time.  
Hence our claim is shown in the case where $k\in (0,1)$. 

\medskip
{\bf (ii)} We assume $k =1$. 
Let $\ell$ be a positive constant satisfying $\ell >\frac{n-2}{2}-1$.  
 It follows from Lemma \ref{lemm:bounduvk} {\bf (ii)} that
\begin{eqnarray*}
(\ell+2) \int_\Omega \left\{(1-\Delta_N)^{-1} \frac{u}{v}\right\} w^{\ell+1}\,dx
 \leq (\ell+2)d_2 \int_\Omega  \left( 1 + \log \frac{w}{w_\ast} \right) w^{\ell+1}\,dx.
\end{eqnarray*}
Since there exists a positive constant $C$ such that 
\[
 1 + \log s \leq s^{\frac{1}{n}} + C \quad \mbox{ for } s \geq 0,
\]
we have
\[
 1 + \log \frac{w}{w_\ast} \leq \left( \frac{w}{w_\ast} \right)^{\frac{1}{n}} + C \left(\frac{w}{w_\ast} \right)^{\frac{1}{n}}.
\]
Hence
$$
(\ell+2) \int_\Omega \left\{(1-\Delta_N)^{-1} \frac{u}{v}\right\} w^{\ell+1}\,dx
\leq
(\ell+2)d_2 (1+C)\left( \frac{1}{w_\ast} \right)^{\frac{1}{n}} 
\int_\Omega w^{\ell+1+\frac{1}{n}}\,dx.
$$
Combining the above with Lemma \ref{lem:auxLp} we have
 \begin{eqnarray*}
 \frac{d}{dt} \int_\Omega w^{\ell+2}\, dx 
+ \frac{2\ell(\ell+2) d_1}{(\ell+1)^2K_{Sob}} 
 |\Omega|^{-\frac{\ell+1}{\ell+2}+\frac{n-2}{n}} 
\left( \int_\Omega w^{\ell+2} dx \right)^{\frac{\ell+1}{\ell+2}}  
\leq {\bf I}_2
\end{eqnarray*}
with 
\begin{eqnarray*}
{\bf I}_2 = (\ell+2)d_2 (1+C)\left( \frac{1}{w_\ast} \right)^{\frac{1}{n}} 
\int_\Omega   w^{\ell+1+\frac{1}{n}}\,dx
- \frac{2\ell(\ell+2) d_1}{(\ell+1)^2K_{Sob}} 
\left( \io w^{(\ell+1)\frac{n}{n-2}}\,dx \right)^\frac{n-2}{n}.
\end{eqnarray*}
Since $1+\frac{1}{n}<\frac{n}{n-2}$,  
we can apply Lemma \ref{lem:key} with sufficiently small $\varepsilon>0$ satisfying 
$$
(\ell+2)d_2 (1+C)\left( \frac{1}{w_\ast} \right)^{\frac{1}{n}}  \cdot \varepsilon
 \leq \frac{2\ell(\ell+2) d_1}{(\ell+1)^2K_{Sob}} 
$$
to have
$$
{\bf I}_2 \leq C_2
$$
with some $C_2>0$. 
Therefore we have
 \begin{eqnarray*} 
 \frac{d}{dt} \int_\Omega w^{\ell+2}\, dx 
+ \frac{2\ell(\ell+2) d_1}{(\ell+1)^2K_{Sob}} 
 |\Omega|^{-\frac{\ell+1}{\ell+2}+\frac{n-2}{n}} 
\left( \int_\Omega w^{\ell+2}\, dx \right)^{\frac{\ell+1}{\ell+2}}  
\leq C_2.
\end{eqnarray*}
We deduce from the above inequality that 
\begin{eqnarray*}
\int_\Omega w(x,t)^{\ell+2} dx & \leq & \max \left\{ \int_\Omega w(x,0)^{\ell+2} dx, \ C_2^{\frac{\ell+2}{\ell+1}} 
\left(
\frac{2\ell(\ell+2) d_1}{(\ell+1)^2K_{Sob}} 
 |\Omega|^{-\frac{\ell+1}{\ell+2}+\frac{n-2}{n}} 
\right)^{-\frac{\ell+2}{\ell+1}} \right\}.
\end{eqnarray*}
Proceeding the same discussion as in {\bf (i)}, 
we derive that for any $p \geq 1$ $L^p$-norms of the function $v$ and $w$ are uniformly bounded in time.  
Then our claim is shown in the case where $k=1$. 

\medskip
{\bf (iii)} We assume $k \in (1,\frac{n}{n-2})$.  
Let $\ell$ be a positive constant satisfying $\ell >\frac{n-2}{2}k-1$.  
Lemma \ref{lemm:bounduvk} {\bf (iii)} leads us to
\begin{eqnarray*}
(\ell+k+1) \int_\Omega \left\{(1-\Delta_N)^{-1} \frac{u}{v^k}\right\} w^{\ell+k}\,dx
&\leq& \frac{d_2^k (\ell+k+1) }{k-1}
\int_\Omega  \left( \frac{k}{w_\ast^{k-1}} - \frac{1}{w^{k-1}}\right)w^{\ell+k}\,dx \nonumber \\
&\leq&  \frac{d_2^k(\ell+k+1)k}{(k-1)w_\ast^{k-1}} \int_\Omega  w^{\ell+k}\,dx.
\end{eqnarray*}  
Combining the above with Lemma \ref{lem:auxLp} we have
 \begin{eqnarray*}
 \frac{d}{dt} \int_\Omega w^{\ell+k+1}\, dx 
+ \frac{2\ell(\ell+k+1) d_1^k}{(\ell+1)^2K_{Sob}} 
 |\Omega|^{-\frac{\ell+1}{\ell+k+1}+\frac{n-2}{n}} 
\left( \int_\Omega w^{\ell+k+1} \,dx \right)^{\frac{\ell+1}{\ell+k+1}}  
\leq {\bf I}_3
\end{eqnarray*}
with 
\begin{eqnarray*}
{\bf I}_3 =  \frac{d_2^k(\ell+k+1)k}{(k-1)w_\ast^{k-1}} \int_\Omega  w^{\ell+k}\,dx
- \frac{2\ell(\ell+k+1) d_1^k}{(\ell+1)^2K_{Sob}} 
\left( \io w^{(\ell+1)\frac{n}{n-2}}\,dx \right)^\frac{n-2}{n}.
\end{eqnarray*}
Since $k <\frac{n}{n-2}$, we can apply Lemma \ref{lem:key} with sufficiently small $\varepsilon>0$ to have
$$
{\bf I}_3 \leq C_3
$$
with some $C_3>0$. 
Therefore we have
 \begin{eqnarray*} 
 \frac{d}{dt} \int_\Omega w^{\ell+k+1}\, dx 
+ \frac{2\ell(\ell+k+1) d_1^k}{(\ell+1)^2K_{Sob}} 
 |\Omega|^{-\frac{\ell+1}{\ell+k+1}+\frac{n-2}{n}} 
\left( \int_\Omega w^{\ell+k+1} \,dx \right)^{\frac{\ell+1}{\ell+k+1}}  
\leq C_3,
\end{eqnarray*}
and then for $t\geq0$,
\begin{eqnarray*}
&&\int_\Omega w(x,t)^{\ell+k+1}\, dx \\
&& \leq \max \left\{ \int_\Omega w(x,0)^{\ell+k+1} \,dx, \ C_3^{\frac{\ell+k+1}{\ell+1}} 
\left(
 \frac{2\ell(\ell+k+1) d_1^k}{(\ell+1)^2K_{Sob}} 
 |\Omega|^{-\frac{\ell+1}{\ell+k+1}+\frac{n-2}{n}} 
\right)^{-\frac{\ell+k+1}{\ell+1}} \right\}.
\end{eqnarray*}
 Proceeding the same discussion as in {\bf (i)},  we complete the proof in the case {\bf (iii)}. 
\end{proof}

%
%\cred{%
\begin{rema}\label{remA1-3}
Lemma \ref{lemm:boundvLp} holds true when we assume {\bf (A1)$^\prime$} instead of {\bf (A1)}.  
In light of Remark \ref{remA-2}, we need some modifications of choice of coefficients.
\end{rema}
%}%

The following proposition asserts boundedness of $v$ and $w$, which follows from the previous lemma.

\begin{prop}
 \label{prop:boundvLinfty}
The functions $v$ and $w$ are bounded in $\Omega \times [0,\infty)$, i.e.,
\[
v^\ast := \sup_{(x,t) \in \Omega \times [0,\infty)} v(x,t)< \infty
\quad \mbox{and}\quad
w^\ast := \sup_{(x,t) \in \Omega \times [0,\infty)} w(x,t) < \infty. 
\]
\end{prop}
In order to show Proposition \ref{prop:boundvLinfty}, we prepare two lemmas.
\begin{lemm}\label{lem:prelimi_iteration}
Let $\hat{p} > 2^\ast := \frac{2n}{n-2}$ and $\hat{k} \in [0, \frac{2}{n}\hat{p})$.  
Set
$$q:= \frac{1}{2}  \hat{p} + \frac{n\hat{k}}{4} \in \left(\frac{n}{n-2},\hat{p}\right).$$
Then for any $\eta \in C^1(\Omega)$ and $\varepsilon_1, \varepsilon_2>0$, it follows
\begin{eqnarray*}
\varepsilon_1 \int_\Omega \eta^{\frac{2\hat{p}}{\hat{p}-\hat{k}}} \,dx  \nonumber 
 \leq 
\varepsilon_2
  \int_\Omega \left( |\nabla \eta|^2 +\eta^2 \right)\, dx 
+ 
\varepsilon_1^{\frac{\hat{p}-\hat{k}}{\hat{p}-\hat{k}-\hat{p}\alpha}}
\left( \frac{K_{Sob}}{\varepsilon_2}\right)^{\frac{\hat{p}\alpha}{\hat{p}-\hat{k}-\hat{p}\alpha}}
\left( \int_\Omega \eta^{\frac{2q}{\hat{p}-\hat{k}}} dx \right)^2
\end{eqnarray*}
where $K_{Sob}$ is the constant in \eqref{eqn:sobevI} and
$$
\alpha :=\frac{(\hat{p}-\hat{k})(\hat{p}-q)}{\hat{p}(\hat{p}-\hat{k}-\frac{2q}{2^\ast})}  \in (0,1).
$$
\end{lemm}
\begin{proof}
We first confirm $\alpha \in (0,1)$.  
By the definition, $\hat{p}-\hat{k}>0$ and  $\hat{p}-q>0$. 
Moreover the direct calculation implies
\begin{eqnarray*}
\hat{p}-\hat{k} - \frac{2q}{2^\ast} &  = & \hat{p}-\hat{k}- \frac{n-2}{2n} \left( \hat{p} + \frac{n\hat{k}}{2} \right) 
\\
& = & \frac{n+2}{2n} \hat{p} -\frac{n+2}{4} \hat{k} \\
& > &0,
\end{eqnarray*} 
which implies $\alpha>0$. 
On the other hand,  it follows from $\hat{k} < \frac{2}{n}\hat{p}$ that
\begin{eqnarray*}
q - \left(\hat{k} + \frac{2q}{2^\ast}  \right)& = & \left(1-\frac{n-2}{n} \right)q - \hat{k} \\
& = & \frac{1}{n} \left( \hat{p} + \frac{n\hat{k}}{2} \right) - \hat{k} \\
& = & \frac{1}{n} \left( \hat{p} - \frac{n\hat{k}}{2} \right)  \\
& > &   0,
\end{eqnarray*}
and then
\begin{eqnarray}\label{calalpha}
\frac{\hat{p}-q}{\hat{p}-\hat{k}-\frac{2q}{2^\ast}} <1.
\end{eqnarray}
Therefore $\alpha <1$. 

Next we will show the desired inequality.  
Since \eqref{calalpha} implies
\begin{eqnarray}\label{parameter}
\nn
\frac{\hat{p}\alpha}{\hat{p}-\hat{k}} 
& = & \frac{\hat{p}}{\hat{p}-\hat{k}} \cdot  \frac{(\hat{p}-\hat{k})(\hat{p}-q)}{\hat{p}(\hat{p}-\hat{k} - \frac{2q}{2^\ast})} \\
\nn
& = &   \frac{\hat{p}-q}{\hat{p}-\hat{k} - \frac{2q}{2^\ast}}\\
&<&1,
\end{eqnarray}
it follows
\begin{eqnarray*}
\frac{2\hat{p}\alpha}{\hat{p}-\hat{k}} < 2 < 2^\ast.
\end{eqnarray*}
Hence we can invoke the H\"older inequality to see that 
\begin{eqnarray*}
\int_\Omega \eta^{\frac{2\hat{p}}{\hat{p}-\hat{k}}} \,dx 
& = & \int_\Omega \eta^{\frac{2\hat{p}\alpha}{\hat{p}-\hat{k}}} \eta^{\frac{2\hat{p}(1-\alpha)}{\hat{p}-\hat{k}}} dx \\
& \leq & \left( \int_\Omega \eta^{2^\ast} dx \right)^{\frac{2\hat{p}\alpha}{(\hat{p}-\hat{k})2^\ast}}  
\left( \int_\Omega \eta^{\frac{2\hat{p}(1-\alpha)2^\ast}{(\hat{p}-\hat{k})2^\ast -2\hat{p}\alpha}} dx \right)^{\frac{(\hat{p}-\hat{k})2^\ast -2\hat{p}\alpha}{(\hat{p}-\hat{k})2^\ast}}.
\end{eqnarray*}
Here the direct calculation guarantees
\begin{eqnarray*}
\frac{2\hat{p}(1-\alpha)2^\ast}{(\hat{p}-\hat{k})2^\ast -2\hat{p}\alpha} 
& =& \frac{2\hat{p}(1-\alpha)2^\ast}{(\hat{p}-\hat{k})2^\ast -2\hat{p} \frac{(\hat{p}-\hat{k})(\hat{p}-q)}{\hat{p}(\hat{p}-\hat{k}- 2q/2^\ast)}} \\
& =& \frac{2\hat{p}2^\ast \left(1-\frac{(\hat{p}-\hat{k})(\hat{p}-q)}{\hat{p}(\hat{p}-\hat{k}-2q/2^\ast )} \right)}{(\hat{p}-\hat{k})\left(2^\ast -\frac{2(\hat{p}-q)}{\hat{p}-\hat{k}-2q/2^\ast} \right)} \\
& =& \frac{2\cdot 2^\ast}{\hat{p}-\hat{k}} \cdot \frac{\hat{p}(\hat{p}-\hat{k}-2q/2^\ast )-(\hat{p}-\hat{k})(\hat{p}-q)}{2^\ast(\hat{p}-\hat{k}-2q/2^\ast )-2(\hat{p}-q)}  \\
& =& \frac{2\cdot 2^\ast}{\hat{p}-\hat{k}} \cdot \frac{q(\hat{p}-\hat{k})-2\hat{p}q/2^\ast}{(2^\ast-2)\hat{p}-2^\ast\hat{k}} \\
& = & \frac{2q}{\hat{p}-\hat{k}},
\end{eqnarray*}
and by applying the Sobolev inequality \eqref{eqn:sobevI} we have
\begin{eqnarray*}
\int_\Omega \eta^{\frac{2\hat{p}}{\hat{p}-\hat{k}}} \,dx 
& \leq & \left( \int_\Omega \eta^{2^\ast} dx \right)^{\frac{2\hat{p}\alpha}{(\hat{p}-\hat{k})2^\ast}}  
\left( \int_\Omega \eta^{\frac{2q}{\hat{p}-\hat{k}}} \,dx \right)^{\frac{(\hat{p}-\hat{k})2^\ast -2\hat{p}\alpha}{(\hat{p}-\hat{k})2^\ast}} \\
& \leq & K_{Sob}^{\frac{\hat{p}\alpha}{\hat{p}-\hat{k}}}
 \left( \int_\Omega \left( |\nabla \eta|^2 +\eta^2 \right) \,dx \right)^{\frac{\hat{p}\alpha}{\hat{p}-\hat{k}}}  
\left( \int_\Omega \eta^{\frac{2q}{\hat{p}-\hat{k}}} \,dx \right)^{\frac{(\hat{p}-\hat{k})2^\ast -2\hat{p}\alpha}{(\hat{p}-\hat{k})2^\ast}}.
\end{eqnarray*}
Due to \eqref{parameter} we can check
$$
\frac{\hat{p}-\hat{k}}{\hat{p}\alpha}>1\quad\mbox{and}\quad
\frac{\hat{p}-\hat{k}}{\hat{p}-\hat{k}-\hat{p}\alpha} >1, 
$$
and by the Young inequality it follows 
\begin{eqnarray*}
&&\varepsilon_1 \int_\Omega \eta^{\frac{2\hat{p}}{\hat{p}-\hat{k}}} \,dx  \nonumber \\
& \leq & 
\varepsilon_2
  \int_\Omega \left( |\nabla \eta|^2 +\eta^2 \right)\, dx \\
&&
+ \frac{\hat{p}-\hat{k}-\hat{p}\alpha}{\hat{p}-\hat{k}} 
\left( \frac{\hat{p}-\hat{k}}{\hat{p}\alpha}  \right)^{-\frac{\hat{p}\alpha}{\hat{p}-\hat{k}-\hat{p}\alpha}}
\varepsilon_1^{\frac{\hat{p}-\hat{k}}{\hat{p}-\hat{k}-\hat{p}\alpha}}
\left( \frac{K_{Sob}}{\varepsilon_2}\right)^{\frac{\hat{p}\alpha}{\hat{p}-\hat{k}-\hat{p}\alpha}}
\left( \int_\Omega \eta^{\frac{2q}{\hat{p}-\hat{k}}} dx \right)^{\frac{(\hat{p}-\hat{k})2^\ast - 2\hat{p}\alpha}{(\hat{p}-\hat{k}-\hat{p}\alpha)2^\ast}}\\
&\leq& 
\varepsilon_2
  \int_\Omega \left( |\nabla \eta|^2 +\eta^2 \right)\, dx 
+ 
\varepsilon_1^{\frac{\hat{p}-\hat{k}}{\hat{p}-\hat{k}-\hat{p}\alpha}}
\left( \frac{K_{Sob}}{\varepsilon_2}\right)^{\frac{\hat{p}\alpha}{\hat{p}-\hat{k}-\hat{p}\alpha}}
\left( \int_\Omega \eta^{\frac{2q}{\hat{p}-\hat{k}}} dx \right)^{\frac{(\hat{p}-\hat{k})2^\ast - 2\hat{p}\alpha}{(\hat{p}-\hat{k}-\hat{p}\alpha)2^\ast}}
\end{eqnarray*}
for any $\varepsilon_1, \varepsilon_2>0$. Finally, we can directly check
$$
\frac{(\hat{p}-\hat{k})2^\ast - 2\hat{p}\alpha}{(\hat{p}-\hat{k}-\hat{p}\alpha)2^\ast}=2
$$
since
\begin{eqnarray*}
\frac{(\hat{p}-\hat{k})2^\ast - 2\hat{p}\alpha}{(\hat{p}-\hat{k}-\hat{p}\alpha)2^\ast} 
& = & 
\frac{ (\hat{p}-\hat{k})2^\ast - 2\hat{p} \frac{(\hat{p}-\hat{k})(\hat{p}-q)}{\hat{p}(\hat{p}-\hat{k}-2q/2^\ast)}}{(\hat{p}-\hat{k})2^\ast - \hat{p}2^\ast \frac{(\hat{p}-\hat{k})(\hat{p}-q)}{\hat{p}(\hat{p}-\hat{k}-2q/2^\ast)}  }  \\
& = & 
\frac{2^\ast - 2 \frac{(\hat{p}-q)}{(\hat{p}-\hat{k}-2q/2^\ast)}}{2^\ast - 2^\ast \frac{(\hat{p}-q)}{(\hat{p}-\hat{k}-2q/2^\ast)}  } \\
& = & 
\frac{2^\ast(\hat{p}-\hat{k}-2q/2^\ast) - 2(\hat{p}-q)}{2^\ast(\hat{p}-\hat{k}-2q/2^\ast) - 2^\ast (\hat{p}-q)} \\
& = & 
\frac{2^\ast(\hat{p}-\hat{k}) - 2\hat{p}}{2^\ast(q-\hat{k}) - 2q} \\
\end{eqnarray*}
and 
\begin{eqnarray*}
\left( 2^\ast (\hat{p}-\hat{k}) - 2\hat{p} \right) - 2\left(2^\ast(q-\hat{k}) - 2q \right)
 & = & 
(\hat{p}-2q+\hat{k})2^\ast - 2(\hat{p}-2q ) \\
& = & \frac{4}{n-2} (\hat{p}-2q) + \hat{k} \frac{2n}{n-2} \\ 
& = & \frac{4}{n-2} \left(- \frac{n\hat{k}}{2} \right) + \hat{k} \frac{2n}{n-2} \\
& = & 0.
\end{eqnarray*}
We complete the proof.
\end{proof}

\begin{lemm}
\label{lemm:odew}
Let $p > \frac{2n}{n-2}+1$ and
$$
M := \sup_{t \geq 0} \left(\int_\Omega w(x,t)^{\frac{nk}{2}} dx \right)^{\frac{2}{n}} +1.
$$
\begin{enumerate}
\item[{\bf (i)}] If $k \in (0,1)$, the following inequality holds:
\begin{eqnarray*}
&&\frac{d}{dt} \int_\Omega w^p dx + \frac{d_1^k}{K_{Sob}M} \int_\Omega w^p \,dx\\ 
 & \leq & 
p^{\frac{p-k}{p-k-(p-k)\alpha}}
\left( \frac{d_2^k}{1-k} \right)^{\frac{p-k}{p-k-(p-k)\alpha}}
\left(\frac{K_{Sob}}{d_1^k}\right)^{\frac{(p-k)\alpha}{p-k-(p-k)\alpha}}
\left( \int_\Omega w^q\, dx \right)^2,
\end{eqnarray*}
where 
\[
q := \frac{1}{2}(p-k)
\quad \mbox{and}\quad
\alpha : = \frac{p-k-q}{p-k-\frac{2q}{2^\ast}}.  
\]
\item[{\bf (ii)}] If $k=1$, the following inequality holds:
\begin{eqnarray*}
&&\frac{d}{dt} \int_\Omega w^p dx + \frac{d_1}{K_{Sob}M} \int_\Omega w^p \,dx\\ 
 & \leq & 
p^{\frac{p-1}{p-1-(p-\frac{n-4}{n-2})\alpha}}
\left( d_2 C \right)^{\frac{p-1}{p-1-(p-\frac{n-4}{n-2})\alpha}}
\left(\frac{K_{Sob}}{d_1}\right)^{\frac{(p-\frac{n-4}{n-2})\alpha}{p-1-(p-\frac{n-4}{n-2})\alpha}}
\left( \int_\Omega w^q\, dx \right)^2,
\end{eqnarray*}
where $C>0$ is some constant depending on $n$ and $w_\ast$, and  
\[
q := \frac{p}{2}-1
\quad \mbox{and}\quad
\alpha := \frac{(p-1)(p-\frac{n-4}{n-2}-q)}{(p-\frac{n-4}{n-2})(p-1-\frac{2q}{2^\ast})}.  
\]
\item[{\bf (iii)}] If $k \in (1,\frac{n}{n-2})$, the following inequality holds:
\begin{eqnarray*}
&&\frac{d}{dt} \int_\Omega w^p dx + \frac{d_1^k}{K_{Sob}M} \int_\Omega w^p \,dx\\ 
 & \leq & 
p^{\frac{p-k}{p-k-(p-1)\alpha}}
\left(\frac{d_2^k k}{(k-1)w_\ast^{k-1}} \right)^{\frac{p-k}{p-k-(p-1)\alpha}}
\left(\frac{K_{Sob}}{d_1^k}\right)^{\frac{(p-1)\alpha}{p-k-(p-1)\alpha}}
\left( \int_\Omega w^q\, dx \right)^2,
\end{eqnarray*}
where 
\[
q := \frac{1}{2}(p-1) + \frac{n}{4}(k-1)
\quad \mbox{and}\quad
\alpha := \frac{(p-k)(p-1-q)}{(p-1)(p-k-\frac{2q}{2^\ast})}.  
\]
\end{enumerate}
\end{lemm}
\begin{proof}
Let $p>\frac{2n}{n-2}+1$.  We first recall Lemma \ref{lem:preLp}:
\begin{equation*}
 \frac{d}{dt} \int_\Omega w^p dx + \frac{4p(p-k-1)d_1^k}{(p-k)^2} \int_\Omega \left(|\nabla w^{\frac{p-k}{2}}|^2 + w^{p-k} \right) dx
  \leq p\int_\Omega \left\{ (1-\Delta_N)^{-1} \frac{u}{v^k} \right\} w^{p-1}\,dx
\end{equation*}
and thus
\begin{equation}
\label{eqn:odineqwI}
 \frac{d}{dt} \int_\Omega w^p dx + 2d_1^k \int_\Omega \left(|\nabla w^{\frac{p-k}{2}}|^2 + w^{p-k} \right) dx
  \leq p\int_\Omega \left\{ (1-\Delta_N)^{-1} \frac{u}{v^k} \right\} w^{p-1}\,dx,
\end{equation}
since 
\[
\frac{4p(p-k-1)}{(p-k)^2} > 2 \Longleftrightarrow (p-1)^2-1> k^2
\quad \mbox{ for } p > \frac{2n}{n-2} + 1 \ \mbox{ and } \ k \in \left(0,\frac{n}{n-2}\right).
\]
Moreover it follows from the H\"older inequality and the Sobolev inequality \eqref{eqn:sobevI} that 
\begin{eqnarray}
\label{eqn:sobolevIV}
\int_\Omega w^p \,dx & \leq & \left( \int_\Omega w^{\frac{p-k}{2} \frac{2n}{n-2}} \,dx \right)^{\frac{n-2}{n}} 
\left( \int_\Omega w^{\frac{nk}{2}} dx \right)^{\frac{2}{n}} \nonumber \\
& \leq & K_{Sob} \left\{ \int_\Omega \left( |\nabla w^{\frac{p-k}{2}}|^2 + w^{p-k} \right) \,dx \right\} M.
\end{eqnarray}
Since Lemma \ref{lemm:boundvLp} guarantees boundedness of $M<\infty$, 
by \eqref{eqn:odineqwI} and \eqref{eqn:sobolevIV} we have
\begin{eqnarray}\label{eqn:odineqwbase}
\nn 
 \frac{d}{dt} \int_\Omega w^p \,dx 
 &+& \frac{d_1^k}{K_{Sob}M} \int_\Omega w^p \,dx \\
 &+& d_1^k \int_\Omega \left(|\nabla w^{\frac{p-k}{2}}|^2 + w^{p-k} \right) \,dx
  \leq p\int_\Omega \left\{ (1-\Delta_N)^{-1} \frac{u}{v^k} \right\} w^{p-1}\,dx.
\end{eqnarray}

\medskip
We will first show {\bf (iii)}. We assume that $k \in (1,\frac{n}{n-2})$.
Since Lemma \ref{lemm:bounduvk} implies
\begin{eqnarray*}
p\int_\Omega \left\{ (1-\Delta_N)^{-1} \frac{u}{v^k} \right\} w^{p-1}\,dx 
 \leq   \frac{p d_2^k k}{(k-1)w_\ast^{k-1}} \int_\Omega  w^{p-1}\,dx,
 \end{eqnarray*}
 we have
\begin{eqnarray}
\label{eqn:odineqwIII}
\nn 
 \frac{d}{dt} \int_\Omega w^p \,dx 
 &+&\frac{d_1^k}{K_{Sob} M}  \int_\Omega w^p \,dx \\
 &+& d_1^k \int_\Omega \left(|\nabla w^{\frac{p-k}{2}}|^2 + w^{p-k} \right) \,dx
  \leq  \frac{p d_2^k k}{(k-1)w_\ast^{k-1}} \int_\Omega  w^{p-1}\,dx.
\end{eqnarray}
Here we can apply Lemme \ref{lem:prelimi_iteration} with
\begin{eqnarray*}
\hat{p}&  =&p-1, \qquad \hat{k} =k-1,\quad
q =\frac{1}{2}  \hat{p} + \frac{n\hat{k}}{4} = \frac{1}{2}(p-1) + \frac{n}{4}(k-1),\\[5mm]\\ 
\eta &=& w^{\frac{\hat{p}-\hat{k}}{2}}=w^{\frac{p-k}{2}},\quad
\alpha =\frac{(\hat{p}-\hat{k})(\hat{p}-q)}{\hat{p}(\hat{p}-\hat{k}-\frac{2q}{2^\ast})} 
= \frac{(p-k)(p-1-q)}{(p-1)(p-k-\frac{2q}{2^\ast})},\\[5mm]
\varepsilon_1 &=&  \frac{p d_2^k k}{(k-1)w_\ast^{k-1}},\quad
\varepsilon_2=d_1^k,
\end{eqnarray*}
since we can check
$\hat{p} >  \frac{2n}{n-2}$ and $\hat{k} \in (0, \frac{2}{n}\hat{p})$
by $p> \frac{2n}{n-2} +1$ and $k \in (1, \frac{n}{n-2})$. 
Then it follows
\begin{eqnarray}
\label{eqn:odineqwIV}
 \frac{p d_2^k}{(k-1)w_\ast^{k-1}} \int_\Omega  w^{p-1} dx 
& \leq&  d_1^k
\left\{ \int_\Omega \left( |\nabla w^{\frac{p-k}{2}}|^2 + w^{p-k} \right) dx \right\} \nonumber \\
&&
+
\left(\frac{pd_2^k k}{(k-1)w_\ast^{k-1}} \right)^{\frac{\hat{p}-\hat{k}}{\hat{p}-\hat{k}-\hat{p}\alpha}}
\left(\frac{K_{Sob}}{d_1^k}\right)^{\frac{\hat{p}\alpha}{\hat{p}-\hat{k}-\hat{p}\alpha}}
\left( \int_\Omega w^q\, dx \right)^2
\end{eqnarray}
where we used
$$
\eta^{\frac{2\hat{p}}{\hat{p}-\hat{k}}} =w^{p-1},\quad
\eta^{\frac{2q}{\hat{p}-\hat{k}}} =w^{q}.
$$

It follows from \eqref{eqn:odineqwIII} and \eqref{eqn:odineqwIV} that
\begin{eqnarray*}
\frac{d}{dt} \int_\Omega w^p \,dx &+& \frac{d_1^k}{K_{Sob}M} \int_\Omega w^p \,dx\\ 
 & \leq & 
\left(\frac{pd_2^k k}{(k-1)w_\ast^{k-1}} \right)^{\frac{\hat{p}-\hat{k}}{\hat{p}-\hat{k}-\hat{p}\alpha}}
\left(\frac{K_{Sob}}{d_1^k}\right)^{\frac{\hat{p}\alpha}{\hat{p}-\hat{k}-\hat{p}\alpha}}
\left( \int_\Omega w^q\, dx \right)^2.
\end{eqnarray*}
Then our claim is shown in the case where $k \in (1,\frac{n}{n-2})$.
 
 \medskip
{\bf (ii)}. We assume that $k=1$. 
Lemma \ref{lemm:bounduvk} (ii) implies that 
$$
p\int_\Omega \left\{ (1-\Delta_N)^{-1} \frac{u}{v^k} \right\} w^{p-1}\,dx
\leq
p d_2 \int_\Omega  \left( 1 + \log \frac{w}{w_\ast} \right) w^{p-1}\,dx
$$
and also there exists a positive constant $C$ such that
\[
\int_\Omega \left( 1 + \log \frac{w}{w_\ast} \right) w^{p-1} dx
 \leq C \int_\Omega w^{p-1+\frac{2}{n-2}} \,dx.
\]
Combining the above with \eqref{eqn:odineqwbase} we have
\begin{eqnarray}
\label{eqn:odineqwV}
\nn 
 \frac{d}{dt} \int_\Omega w^p \,dx 
 &+& \frac{d_1}{K_{Sob}M} \int_\Omega w^p \,dx \\
 &+& d_1 \int_\Omega \left(|\nabla w^{\frac{p-1}{2}}|^2 + w^{p-1} \right) \,dx
 \leq  p d_2 C\int_\Omega w^{p-1+\frac{2}{n-2}} \,dx.
 \end{eqnarray}
Here we can apply Lemme \ref{lem:prelimi_iteration} with
\begin{eqnarray*}
\hat{p}&  =&p-1+\frac{2}{n-2},  \qquad \hat{k} =\frac{2}{n-2},\quad
q =\frac{1}{2}  \hat{p} + \frac{n\hat{k}}{4}
 = \frac{p}{2} + \frac{2}{n-2},\\[5mm]\\ 
\eta &=& w^{\frac{\hat{p}-\hat{k}}{2}}=w^{\frac{p-1}{2}},\quad
\alpha =\frac{(\hat{p}-\hat{k})(\hat{p}-q)}{\hat{p}(\hat{p}-\hat{k}-\frac{2q}{2^\ast})} 
= \frac{(p-1)(p-\frac{n-4}{n-2}-q)}{(p-\frac{n-4}{n-2})(p-1-\frac{2q}{2^\ast})},\\[5mm]
\varepsilon_1 &=& p d_2 C,\quad
\varepsilon_2=d_1,
\end{eqnarray*}
then the argument similar to that to establish our claim in the case where $k \in (1,n/(n-2))$ leads us to
\begin{eqnarray*}
\frac{d}{dt} \int_\Omega w^p \,dx &+& \frac{d_1}{K_{Sob}M} \int_\Omega w^p \,dx\\ 
 & \leq & 
\left( p d_2 C \right)^{\frac{\hat{p}-\hat{k}}{\hat{p}-\hat{k}-\hat{p}\alpha}}
\left(\frac{K_{Sob}}{d_1}\right)^{\frac{\hat{p}\alpha}{\hat{p}-\hat{k}-\hat{p}\alpha}}
\left( \int_\Omega w^q\, dx \right)^2.
\end{eqnarray*}
Our claim is shown in the case where $k=1$.

 \medskip
{\bf (i)}. We assume that $k \in (0,1)$. 
Combining \eqref{eqn:odineqwbase} with Lemma \ref{lemm:bounduvk} imply that 
\begin{eqnarray*}
 \frac{d}{dt} \int_\Omega w^p \,dx 
 &+& \frac{d_1^k}{K_{Sob}M} \int_\Omega w^p \,dx \\
 &+& d_1^k \int_\Omega \left(|\nabla w^{\frac{p-k}{2}}|^2 + w^{p-k} \right) \,dx
  \leq   \frac{p d_2^k}{(1-k)} \int_\Omega  w^{p-k} dx.
 \end{eqnarray*}
Here we can apply Lemme \ref{lem:prelimi_iteration} with
\begin{eqnarray*}
\hat{p}&  =&p-k, \qquad \hat{k} =0,\quad
q =\frac{1}{2}  \hat{p} + \frac{n\hat{k}}{4}
 =  \frac{p}{2} - \frac{k}{2},\\[5mm]\\ 
\eta &=& w^{\frac{\hat{p}-\hat{k}}{2}}=w^{\frac{p-k}{2}},\quad
\alpha =\frac{(\hat{p}-\hat{k})(\hat{p}-q)}{\hat{p}(\hat{p}-\hat{k}-\frac{2q}{2^\ast})} 
= \frac{p-k-q}{p-k-\frac{2q}{2^\ast}},\\[5mm]
\varepsilon_1 &=& \frac{p d_2^k}{(1-k)},\quad
\varepsilon_2=d_1^k,
\end{eqnarray*}
to have
\begin{eqnarray*}
&&\frac{d}{dt} \int_\Omega w^p dx + \frac{d_1^k}{K_{Sob}M} \int_\Omega w^p \,dx\\ 
 & \leq & 
\left( p \frac{d_2^k}{(1-k)} \right)^{\frac{\hat{p}-\hat{k}}{\hat{p}-\hat{k}-\hat{p}\alpha}}
\left(\frac{K_{Sob}}{d_1^k}\right)^{\frac{\hat{p}\alpha}{\hat{p}-\hat{k}-\hat{p}\alpha}}
\left( \int_\Omega w^q\, dx \right)^2.
\end{eqnarray*}
Then our claim is shown in the case where $k \in (0,1)$.
Thus the proof is complete.  
\end{proof}
 
 We are in the position to give the proof of Proposition \ref{prop:boundvLinfty}. 
\begin{proof}[Proof of Proposition \ref{prop:boundvLinfty}]
The boundedness of the function $w$ is obtained by an argument similar to that to establish \cite[Lemma 3.5]{FJ21} and Lemma \ref{lemm:odew}.
We first consider the case where $k \in (1,\frac{n}{n-2})$, we shall show boundedness of the function $w$. 
From (iii) of Lemma \ref{lemm:odew} and Lemma \ref{lemm:boundvLp} it follows that 
\begin{eqnarray*}
\label{eqn:boundp1}
\lefteqn{\sup_{t\geq 0} \left(\int_\Omega w(x,t)^p \,dx\right)} \nonumber \\
 & \leq & \int_\Omega w_0^p \,dx \\
 \nn
& +& \frac{K_{Sob}M}{d_1^k}
 p^{\frac{p-k}{p-k-(p-1)\alpha}}
\left(\frac{d_2^k k}{(k-1)w_\ast^{k-1}} \right)^{\frac{p-k}{p-k-(p-1)\alpha}}
\left(\frac{K_{Sob}}{d_1^k}\right)^{\frac{(p-1)\alpha}{p-k-(p-1)\alpha}}
\left\{ \sup_{t\geq 0} \left(\int_\Omega w(x,t)^{q} dx\right) \right\}^2 \nonumber \\
 & \leq & \frac{\|w_0\|_{L^\infty(\Omega)}^{1-\frac{n}{2}(k-1)}}{|\Omega|} \left(|\Omega|\|w_0\|_{L^\infty(\Omega)}^q \right)^2
\nonumber \\ 
&+& \frac{K_{Sob}M}{d_1^k}
 p^{\frac{p-k}{p-k-(p-1)\alpha}}
\left(\frac{d_2^k k}{(k-1)w_\ast^{k-1}} \right)^{\frac{p-k}{p-k-(p-1)\alpha}}
\left(\frac{K_{Sob}}{d_1^k}\right)^{\frac{(p-1)\alpha}{p-k-(p-1)\alpha}}
\left\{ \sup_{t\geq 0} \left(\int_\Omega w(x,t)^{q} dx\right) \right\}^2,
\end{eqnarray*}
where $p>\frac{2n}{n-2}+1$,
\[
q = \frac{1}{2}(p-1) + \frac{n}{4}(k-1)
\quad \mbox{and}\quad
\alpha = \frac{(p-k)(p-1-q)}{(p-1)(p-k-\frac{2q}{2^\ast})}.  
\]

Let us put 
\[
 D := \max \left\{ \frac{\|w_0\|_{L^\infty(\Omega)}^{1-\frac{n}{2}(k-1)}}{|\Omega|}, 
\frac{K_{Sob}M}{d_1^k},
\frac{d_2^k k}{(k-1)w_\ast^{k-1}},
\frac{K_{Sob}}{d_1^k}
  \right\}
\]
and
\[
 A(p) := \max \left\{ \sup_{t\geq 0} \left(\int_\Omega w(x,t)^p dx\right), |\Omega|\|w_0\|_{L^\infty(\Omega)}^p \right\}.
\]
Those imply that for any $p> \frac{2n}{n-2}+1$,
\begin{eqnarray}
\label{eqn:boundp2}
 A(p) & \leq & 2(1+D)^{1+\frac{(p-1)\alpha +(p-k)}{p-k-(p-1)\alpha}}
p^{\frac{p-k}{p-k-(p-1)\alpha}} A(q)^2.
\end{eqnarray}
Since we see that 
\[
 \lim_{p \rightarrow \infty} \alpha = \frac{1-\frac{1}{2}}{1-\frac{n-2}{2n}} = \frac{n}{n+2}, 
\]
\[
 \lim_{p \rightarrow \infty} \frac{(p-1)\alpha +(p-k)}{p-k-(p-1)\alpha} = \frac{\frac{n}{n+2} +1}{1-\frac{n}{n+2}} = n+1
\]
and
\[
 \lim_{p \rightarrow \infty} \frac{p-k}{p-k-(p-1)\alpha} = \frac{1}{1-\frac{n}{n+2}} = \frac{n+2}{2},
\]
there exists a positive constant $J_0$ such that for $j \geq J_0$
\begin{eqnarray}
\label{eqn:boundp3}
\nn
&& A\left(2^{j+1}  + \frac{n(k-1)}{2} - 1  \right) \\
 & \leq & 2(1+D)^{2(n+1)}
\left(2^{j+1}  +\frac{n(k-1)}{2}-1  \right)^{n} 
A\left(2^{j}  +\frac{n(k-1)}{2}- 1  \right)^2 \nonumber \\
& \leq &  2^{n+1} (1+D)^{2(n+1)}
2^{n(j+1)} A\left(2^{j} +\frac{n(k-1)}{2}-1  \right)^2,
\end{eqnarray}
where in the last line we used 
\begin{eqnarray*}
\left(2^{j+1}  +\frac{n(k-1)}{2}-1  \right)^{n}
&\leq& \left(2^{j+1}  +\frac{2}{n-2}  \right)^{n}\\
&\leq& \left(2 \cdot 2^{j+1}   \right)^{n} = 2^n \cdot 2^{n(j+1)}.
\end{eqnarray*}
%}%
We deduce from (\ref{eqn:boundp3}) that for any $J\in \mathbf{N}$,
\begin{eqnarray}
\label{eqn:boundp4}
 A\left(2^{J+J_0}  +\frac{n(k-1)}{2}-1 \right)
  & \leq & \Pi_{j=1}^J \left\{2^{n+1} (1+D)^{2(n+1)}
2^{n(j+J_0)} \right\}^{2^{J-j}} A\left(2^{J_0} +\frac{n(k-1)}{2}-1  \right)^{2^J}, \nonumber \\
&&
\end{eqnarray}
which together with 
\begin{eqnarray*}
 \Pi_{j=1}^J \left\{2^{n+1}  (1+D)^{2(n+1)}
2^{n(j+J_0)} \right\}^{\frac{2^{J-j}}{2^{J+J_0}}} 
& \leq & \left\{ 2^{(nJ_0+n+1)}  (1+D)^{2(n+1)}2^{2n}\right\}^{2^{-J_0}},
\end{eqnarray*}
we imply that 
\begin{eqnarray*}
&&\limsup_{J \rightarrow \infty} A\left(2^{J+J_0}+ \frac{n(k-1)}{2}- 1  \right)^{2^{-J-J_0}} \nonumber  \\
& \leq & \left\{2^{(nJ_0+1)}(1+D)^{2(n+1)}2^{2n}\right\}^{2^{-J_0}} 
A\left(2^{J_0} +\frac{n(k-1)}{2}-1  \right)^{2^{-J_0}}<\infty,
\end{eqnarray*}
and then
\begin{eqnarray}
\label{eqn:boundp5}
\nn
\sup_{t \geq 0}\|w(\cdot,t)\|_{L^\infty} 
&\leq& \limsup_{J \rightarrow \infty} A\left(2^{J+J_0}-\frac{n(k-1)}{2}+1\right)^{-(2^{J+J_0} + \frac{n(k-1)}{2}-1) }\\
&<&\infty.
\end{eqnarray}
Then we obtain the boundedness of the function $w$ in the case where $k \in (1,\frac{n}{n-2})$.
An argument similar to the one to establish (\ref{eqn:boundp5}) leads to the boundedness of the function $w$ in the case where $k=1$ and $k \in (0,1)$. Thus, we get boundedness of the function $w$ for $k \in (0,\frac{n}{n-2})$, from which together with Lemma  \ref{lemm:Estvw} we imply boundedness of the function $v$ for $k \in (0,\frac{n}{n-2})$. Therefore the proof is complete. 
\end{proof}
 
 %
%\cred{%
\begin{rema}\label{remA1replaced}
Proposition \ref{prop:boundvLinfty} holds true when we assume {\bf (A1)$^\prime$} 
since the above proof is based on Lemma \ref{lemm:boundvLp} and Lemma \ref{lemm:odew}. 
As noted in Remark \ref{remA1-3}, Lemma \ref{lemm:boundvLp} holds true under {\bf (A1)$^\prime$}. 
As to Lemma \ref{lemm:odew}, Lemma \ref{lemm:bounduvk} and Lemma \ref{lem:preLp} play important roles in the proof. 
These lemmas also holds true when we assume {\bf (A1)$^\prime$} as noted in Remark \ref{remA-11} and Remark \ref{remA-2}. After modifying coefficients we can proceed the similar lines of the proof of Lemma \ref{lemm:odew}.
\end{rema}
%}%
%
%
 
 \section{Proof of Theorem  \ref{theo:TG}}
 \label{sectionproof}

Theorem \ref{theo:TG} will be shown by Proposition \ref{prop:boundvLinfty} and an argument similar to the proof of \cite[Theorem 1.1]{FSpreprint}. 
%\cred{
Hereafter we assume {\bf (A1)} (or {\bf (A1)$^\prime$}). 
Let $(u,v)$ be a classical solution of (\ref{eqn:kinemod}) in $\Omega \times (0,\infty)$. %}
We first show that $v$ is uniformly H\"older continuous in $\overline{\Omega} \times [0,\infty)$.

\begin{lemm}
\label{lemm:Holderv} 
Then there exists some $\alpha \in (0,1)$ such that
$$v \in C^{\alpha, \frac{\alpha}{2}} ({\overline{\Omega}\times [0,\infty)}),$$
that is, there exist positive constants $\Lambda$ and $\alpha \in (0,1)$ such that
\begin{equation}
\label{eqn:unifholder}
 |v(x,t) - v(\tilde{x},\tilde{t})| \leq \Lambda (|x-\tilde{x}|^\alpha + |t-\tilde{t}|^{\alpha/2})
\end{equation} 
for any $(x,t)$ and $(\tilde{x},\tilde{t}) \in \overline{\Omega} \times [0,\infty)$.
\end{lemm}
\begin{proof}
As noted in Proposition \ref{prop:TGexist}, $(u,v,w) \in (C^{2,1}(\overline{\Omega}\times(0,\infty)))^3$ satisfies \eqref{eqn:kinemod} and \eqref{eqn:evolutionw} classically on $(0,\infty)$.  
Lemma \ref{lemm:Minvw} and Proposition \ref{prop:boundvLinfty} guarantee that 
\[
 \gamma (v^\ast) \leq \gamma (v) \leq \gamma (v_\ast) \quad \mbox{ in } \overline{\Omega} \times [0,\infty),
\]
and also it follows from the maximal principle and Proposition \ref{prop:boundvLinfty} that 
\begin{eqnarray*}
0\leq  (I-\Delta_N)^{-1}[\gamma(v)u] 
 &\leq& (I-\Delta_N)^{-1}[\gamma(v_\ast)u]\\ 
&=& \gamma(v_\ast) (I-\Delta_N)^{-1}[u]\\
&=& \gamma(v_\ast) w^\ast.
\end{eqnarray*}
Then applying the H\"older regularity estimate (\cite[Corollary 7.51]{Lieberman}) to \eqref{eqn:evolutionw}, we can find some $\alpha\in (0,1)$ such that 
$$w \in C^{\alpha, \frac{\alpha}{2}} ({\overline{\Omega}\times [0,\infty)}).$$
That is to say, the estimate guarantees constants $\tilde{\Lambda}$ and $\alpha \in (0,1)$ satisfying
\[
 |w(x,t) - w(\tilde{x},\tilde{t})| \leq \tilde{\Lambda} (|x-\tilde{x}|^\alpha + |t-\tilde{t}|^{\alpha/2})
\] 
for any $(x,t)$ and $(\tilde{x},\tilde{t}) \in \overline{\Omega} \times [0,\infty)$.

Let $z(x,t):= (1-\Delta_N)^{-1}v(x,t). $
By the definition of $w$ and $z$, it follows
\begin{equation}\label{key_eq2}
z_t = \Delta z -z +w,
\end{equation}
where $w \in C^{\alpha, \frac{\alpha}{2}} ({\overline{\Omega}\times [0,\infty)})$ as proved above. 
By applying the parabolic regularity estimate (\cite[Chapter IV, Theorem 5.3]{LSU})  to \eqref{key_eq2}, we have 
$$z \in C^{2+\alpha, 1+\frac{\alpha}{2}} ({\overline{\Omega}\times [0,\infty)}),$$
where we used the fact $z(0)$ satisfies the compatibility condition. 
Then, by the elliptic regularity estimate it follows
$$v \in C^{\alpha, \frac{\alpha}{2}} ({\overline{\Omega}\times [0,\infty)}).$$
In fact, for $N=0,1,2,\cdots$, let $I_N := (N,N+2)$ and let $v_{1N}$ and $v_{2N}$ be functions satisfying 
\begin{eqnarray*}
&&\frac{\partial}{\partial t} v_{1N} = \Delta v_{1N} - v_{1N} + u \quad \mbox{ in } \Omega \times I_N, \\
&&\frac{\partial }{\partial t} v_{2N} = \Delta v_{2N} - v_{2N} \quad \mbox{ in } \Omega \times I_N, \\
&& \frac{\partial v_{1N}}{\partial \nu} = \frac{\partial v_{2N}}{\partial \nu} = 0  \quad \mbox{ on } \partial\Omega \times I_N, \\
&& v_{1N}(\cdot,N) = 0, \ v_{2N}(\cdot,N)=v(\cdot,N) \quad \mbox{ in } \Omega. 
\end{eqnarray*}
Since the function $z_{1N} := (1-\Delta_N)^{-1} v_{1N}$ satisfies
\begin{eqnarray*}
&&\frac{\partial}{\partial t} z_{1N} = \Delta z_{1N} - z_{1N} + w \quad \mbox{ in } \Omega \times I_N, \\
&& \frac{\partial z_{1N}}{\partial \nu} = 0  \quad \mbox{ on } \partial\Omega \times I_N, \\
&& z_{1N}(\cdot,N) = 0 \quad \mbox{ in } \Omega, 
\end{eqnarray*}
the parabolic regularity estimate (\cite[Chapter IV, Theorem 5.3]{LSU})  to the function $z_{1N}$, we have that 
\[
 |v_{1N}(x,t) - v_{1N}(\tilde{x},\tilde{t})| \leq \Lambda_1 (|x-\tilde{x}|^\alpha + |t-\tilde{t}|^{\alpha/2})
\] 
for any $(x,t)$ and $(\tilde{x},\tilde{t}) \in \overline{\Omega} \times I_N$, where the constant $\Lambda_1$ is independent of $N$.
Moreover, the standard parabolic regularity estimate guarantees that 
\[
 |v_{2N}(x,t) - v_{2N}(\tilde{x},\tilde{t})| \leq \Lambda_2 (|x-\tilde{x}|^\alpha + |t-\tilde{t}|^{\alpha/2})
\]
for any $(x,t)$ and $(\tilde{x},\tilde{t}) \in \overline{\Omega} \times (N+\frac{1}{2},N+2)$, where the constant $\Lambda_2$ depends only on $v^\ast$. Then, we can find the desired estimate, since $v=v_{1N}+v_{2N}$ and $v \in C^{\alpha, \frac{\alpha}{2}} ({\overline{\Omega}\times [0,1]})$. 
We conclude the proof. 
\end{proof}
 
\begin{proof}[Proof of Theorem \ref{theo:TG}]
Define the operator
$$
A(t) \varphi := -\gamma (v(t)) (\Delta \varphi -\varphi)
$$
for $\varphi \in \{ \psi \in  W^{2, p}(\Omega) \,|\, \partial_\nu \psi = 0\ \mbox{on }\partial \Omega\}$ with $p\in (1,\infty)$. 
Since $\gamma (v) $ is H\"older continuous (Lemma \ref{lemm:Holderv}), 
the operator $-A(t)$ generates analytic semigroup in $L^p(\Omega)$ for any $p\in (1,\infty)$. 
For any $N=0,1,2,\cdots$, we can check that $w \in C([N,N+2]; L^p(\Omega)) \cap C^1((N,N+2); L^p(\Omega))$ is a solution of the evolution equation 
\[
\frac{d}{dt} w + A(t)w = F \quad t \in (N,N+2),
\]
where $F = (1-\Delta_N)^{-1} (\gamma (v)u)$. 
Since the elliptic regularity theorem implies $F \in C([0,\infty); L^p(\Omega))$ for any $p \in (1,\infty)$, 
in view of the abstract theory (\cite[Theorem 5.2.2]{Tanabe}), 
the solution $w$ can be represented by the integral equation: 
$$
w(t) = U(t,N) w(N) + \int_N^t U(t,s)F(s)\,ds
\qquad t \in [N,N+2],
$$
where $U(t,s)$ is the fundamental solution. 
Due to the fact that $w(N)$ and $F(s)$ satisfy the Neumann boundary condition, we can apply the estimate of fundamental solutions (\cite[Theorem 5.2.1]{Tanabe}) to have 
\begin{eqnarray*}
&&\| A(t)w(t)\|_{L^p(\Omega)} \\
&\leq & \| A(t) U(t,N) w(N) \|_{L^p(\Omega)}
 + \int_N^t \| A(t) U(t,s)(A(s))^{-1}A(s)F(s)\|_{L^p(\Omega)} \,ds\\
 &\leq & \frac{C(p,v_\ast,v^\ast)}{t-N} \|w(N) \|_{L^p(\Omega)}
 + C(p,v_\ast,v^\ast) \int_N^t  \| A(s)F(s)\|_{L^p(\Omega)} \,ds
\end{eqnarray*}
with some $C(p,v_\ast,v^\ast)>0$. Here, we notice that $C(p,v_\ast,v^\ast)$ is independent of $N$, since $\gamma(v)$ is uniformly H\"older continuous on $\overline{\Omega} \times [0,\infty)$. 
By the definition of $w$, it follows
\begin{eqnarray*}
\| A(t)w(t)\|_{L^p(\Omega)} 
&= & \| \gamma (v(\cdot,t)) (\Delta w(\cdot,t) -w(\cdot,t))\|_{L^p(\Omega)} \\
&=&
\| \gamma (v(\cdot,t))u(\cdot,t)\|_{L^p(\Omega)}.
\end{eqnarray*}
We also obtain for $s \in [N,N+2]$,
\begin{eqnarray*}
A(s)F(s) &=& 
-\gamma (v(\cdot,s)) (\Delta - 1)
\left[(1-\Delta_N)^{-1}(\gamma(v(\cdot,s))u(\cdot,s))\right]\\ 
&=&  (\gamma (v(\cdot,s))^2 u(\cdot,s).
\end{eqnarray*}
For any $p\in (1,\infty)$ it follows for $t\in (N,N+2]$,
$$
 \| u(t) \|_{L^p(\Omega)}  
 \leq 
 \frac{C^\prime(p,v_\ast,v^\ast)}{t-N} w^\ast |\Omega|^{1/p}
 + C^\prime(p,v_\ast,v^\ast) \int_N^t  \| u(s) \|_{L^p(\Omega)} \,ds
$$
with some $C^\prime(p,v_\ast,v^\ast)>0$.
Then, for any $\varepsilon \in (0,\frac{1}{2})$ and $N=1,2,3,\cdots$ we see that for $t \in (N+\varepsilon,N+1+\varepsilon)$,
\begin{eqnarray*}
 \| u(t) \|_{L^p(\Omega)}  
& \leq & 
 \frac{C^\prime(p,v_\ast,v^\ast)}{\varepsilon} w^\ast |\Omega|^{1/p}
 + C^\prime(p,v_\ast,v^\ast) \varepsilon \sup_{N < t < N+\varepsilon} \| u(t) \|_{L^p(\Omega)} \\
&& \qquad + \int_{N+\varepsilon}^t  \| u(s) \|_{L^p(\Omega)} \,ds
\end{eqnarray*}
 and then by Gronwall's lemma it follows
\begin{eqnarray*}
\sup_{N+\varepsilon < t < N+1+\varepsilon} \| u(t) \|_{L^p(\Omega)}  
 \leq 
 \frac{C^{\prime\prime}(p,v_\ast,v^\ast)}{\varepsilon}
 + C^{\prime\prime}(p,v_\ast,v^\ast) \varepsilon \sup_{N-1+\varepsilon < t < N+\varepsilon} \| u(t) \|_{L^p(\Omega)}
\end{eqnarray*}
with some $C^{\prime\prime}(p,v_\ast,v^\ast)>0$. 
Taking $\varepsilon>0$ such that $C^{\prime\prime}(p,v_\ast,v^\ast) \varepsilon \leq 1/2$, we can find a positive constant $C^{\prime\prime\prime}(p,v_\ast,v^\ast)$ such that
\[
 \sup_{t \geq 0} \|u(t)\|_{L^p(\Omega)} \leq C^{\prime\prime\prime}(p,v_\ast,v^\ast).
\]
We can pick up some $p>\frac{n}{2}$ in the above and by Moser's iteration argument (see \cite{Anh19}) it follows
$$
\sup_{t \geq 0} \| u(t)\|_{L^\infty(\Omega)} + \sup_{t \geq 0} \| v(t)\|_{L^\infty(\Omega)} < \infty.
$$
The proof is complete.  
\end{proof}

\section{Discussion}
\label{sec:discussion}
In this section we will discuss about the critical condition which distinguishes boundedness and unboundedness of solutions. 

We first recall the work \cite{YK2017} to compare the condition with Theorem \ref{theo:TG}. 
In \cite[Theorem 2.9]{YK2017} global existence and boundedness of solutions to \eqref{eqn:kinemod} with 
\begin{description}
 \item[(A3)] $\displaystyle \gamma (v) = \frac{c_0}{v^k}$
\end{description}
are established for $\tau=1$, any $k>0$ and sufficiently small $c_0>0$.  
Here we point out that this smallness of $c_0$ depends on the relaxation time $\tau>0$ and the size of initial data.  
Indeed, let $(u,v)$ be a solution of the system:
 \begin{align*}
\begin{cases}
	\displaystyle  u_t= \Delta \left(\frac{ u}{v^k} \right)
		&\mathrm{in}\ \Omega\times(0,\infty), \\[1mm]
	\tau v_t=\Delta v - v + u &\mathrm{in}\ \Omega\times(0,\infty).   
\end{cases}
 \end{align*}
 Putting $(u_\lambda(t), v_{\lambda}(t)):= 
 (\lambda u( \frac{c_0}{\lambda^{k}} t), \lambda v( \frac{c_0}{\lambda^{k}} t) )$ 
 with $\lambda>0$ and $c_0>0$,  
 we can check
 
$$
\left(u_{\lambda}\right)_t (t)
= \frac{c_0}{\lambda^{k-1}}  u_t \left( \frac{c_0}{\lambda^{k}} t\right)
=  \frac{c_0}{\lambda^{k-1}} \Delta \left(\frac{  u}{v^k} \right)\left( \frac{c_0}{\lambda^{k}} t\right)
=   \Delta \left(\frac{ c_0 u_{\lambda} (t)}{\left( v_{\lambda}(t) \right)^k} \right),
$$
and
$$
	\displaystyle	\frac{\tau \lambda^{k}}{c_0} \left(v_{\lambda}\right)_t (t)  =
		\frac{\tau \lambda^{k}}{c_0}  \cdot  \frac{c_0}{\lambda^{k-1}}  v_t \left( \frac{c_0}{\lambda^{k}} t\right) 
		= \lambda \cdot \tau  v_t  \left( \frac{c_0}{\lambda^{k}} t\right)
		= \Delta v_\lambda (t) -  v_\lambda (t) +  u_\lambda (t).
$$
Hence $(u_\lambda, v_{\lambda})$ satisfies
 \begin{align*}
\begin{cases}
	\displaystyle \left(u_{\lambda}\right)_t  = \Delta \left(\frac{ c_0 u_{\lambda}  }{\left(v_{\lambda}  \right)^k} \right)
		&\mathrm{in}\ \Omega\times(0,\infty), \\[1mm]
	\displaystyle	\frac{\tau \lambda^{k}}{c_0} \left(v_{\lambda}\right)_t   =\Delta v_{\lambda} - v_{\lambda} + u_{\lambda}&\mathrm{in}\ \Omega\times(0,\infty)
\end{cases}
 \end{align*}
 and
  $$
\|u_\lambda (0)\|_{L^1(\Omega)} = \lambda \|u_0\|_{L^1 (\Omega)}.
 $$
From the above scaling argument, smallness of $c_0 $ is equivalent to largeness of the relaxation time $\frac{\tau\lambda^k}{c_0}>0$ and also largeness of $\lambda^k$. 
 Especially for $k>1$, smallness of $c_0$ is equivalent to largeness of the initial data.  Actually, in the proof of \cite[Theorem 2.9]{YK2017} the restriction depends on the size of initial data (see \cite[p.109]{YK2017}; the restriction on $c_0$ depends on the lower bound of $v$, which depends on the initial mass).
On the other hand, we emphasize that Theorem \ref{theo:TG} holds true for any $\tau>0$ and any size of initial data if $k \in (0, \frac{n}{n-2})$. 
Therefore Theorem \ref{theo:TG}, \cite{FSpreprint} and \cite[Theorem 2.9]{YK2017} can be summarized as follows: let $\tau=1$ and $\gamma (v) = v^{-k}$. 
\begin{itemize}
\item For any $k>0$ and any regular initial data,  solutions of \eqref{eqn:kinemod} exist globally in time.
\item If $k \in (0, \frac{n}{n-2})$ then for any size of initial data the above global solution is uniformly bounded.
\item If $k \geq \frac{n}{n-2}$ then for large initial data the above global solution is uniformly bounded.  
\end{itemize}

Moreover we can take the similar observation given in \cite{FS2018}, in which 
the critical condition for global existence of the logarithmic Keller--Segel system is discussed.  When we consider the stationary state of the first equation
$$
0 = \Delta (u v^{-k})\qquad  \mathrm{in}\ \Omega\times(0,\infty),
$$
by the 0-Neumann boundary condition it follows
$$
u = C v^{k}
$$
with some $C>0$. Taking the mass conservation law into account, we have
$$
u = \|u_0\|_{L^1(\Omega)} \frac{v^k}{\io v^k\,dx}
$$
and then the system is reduced to the nonlinear heat equation
\begin{align*}
\begin{cases}
v_t = \Delta v - v + \|u_0\|_{L^1(\Omega)} \frac{v^k}{\io v^k\,dx} 
		&\mathrm{in}\ \Omega\times(0,\infty), \\[1mm]
\frac{\partial v }{\partial\nu} =0
&\mathrm{on}\ \partial \Omega\times(0,\infty).
\end{cases}
 \end{align*}
By the similar way in \cite[Theorem 44.5]{QS}, 
radially symmetric blowup solutions of the above equation can be constructed when $k>\frac{n}{n-2}$. 

Based on the above observations we give the following conjecture. 
\begin{conj}
Let $\tau=1$ and $\gamma (v) = v^{-k}$ with $k \geq \frac{n}{n-2}$.  
For some small initial data, the corresponding global solution of \eqref{eqn:kinemod} blows at infinite time, that is,
$$
\limsup_{t\to \infty}  \left( \|u(t)\|_{L^\infty(\Omega)} + \|v(t)\|_{L^{\infty}(\Omega)} \right)=\infty.
$$
\end{conj}

%
%
%
%\cred{
\noindent\textbf{Acknowledgments} \\
K. Fujie is supported by Japan Society for the Promotion of Science (Grant-in-Aid for Early-Career Scientists; No.\ 19K14576). 
T. Senba is supported by Japan Society for the Promotion of Science (Grant-in-Aid for Scientific Research(C); No.\ 18K03386)
%}%

\end{document}